\documentclass{article}
\usepackage[body={5.5in,9in}]{geometry}

\usepackage{graphicx} 
\usepackage{amssymb}
\usepackage{amsmath,amsthm,bbm}
\usepackage{xcolor}
\usepackage{cite}
\usepackage{stackengine}

\newtheorem{theorem}{Theorem}[section]  
\newtheorem{lemma}[theorem]{Lemma}  
\newtheorem{assumption}[theorem]{Assumption}  
\newtheorem{corollary}[theorem]{Corollary}  
\newtheorem{proposition}[theorem]{Proposition}
  
\newtheorem{remark}[theorem]{Remark}

\theoremstyle{definition}

\newcommand{\R}{\mathbb{R}}
\newcommand{\M}{\mathbb{M}}
\newcommand{\Sph}{\mathbb{S}}
\newcommand{\Nats}{\mathbb{N}}

\newcommand{\dif}{\mathrm{d}}

\newcommand{\PhiB}{\mathbf{\Phi}}

\newcommand{\MM}{\mathbf{M}}
\newcommand{\MLL}{\MM^{\rm LL}}
\newcommand{\MFD}{\MM^{\rm FD}}
\newcommand{\K}{\mathbf{K}}
\newcommand{\Id}{\mathbf{I}}
\newcommand{\PP}{\mathbf{P}}
\newcommand{\QQ}{\mathbf{Q}}
\newcommand{\opL}{\mathcal{L}}
\renewcommand{\AA}{\mathbf{A}}
\newcommand{\BB}{\mathbf{B}}
\newcommand{\CC}{\mathbf{C}}
\newcommand{\DD}{\mathbf{D}}
\newcommand{\EE}{\mathbf{E}}

\renewcommand{\SS}{\mathbf{S}}
\newcommand{\RR}{\mathbf{R}}
\newcommand{\UU}{\mathbf{U}}
\newcommand{\VV}{\mathbf{V}}
\newcommand{\WW}{\mathbf{W}}
\newcommand{\XX}{\mathbf{X}}

\newcommand{\ZZ}{\mathbf{Z}}

\newcommand{\zero}{\mathbf{0}}

\newcommand{\LLambda}{\boldsymbol{\Lambda}}
\newcommand{\TTheta}{\boldsymbol{\Theta}}
\newcommand{\GGamma}{\boldsymbol{\Gamma}}

\newcommand{\Ran}{\operatorname{Range}}
\newcommand{\Nul}{\operatorname{Null}}

\newcommand{\dist}{\operatorname{dist}}

\newcommand{\NB}{\mathbf{N}}

\newcommand{\diag}{\mathrm{diag}}

\newcommand{\m}{\tilde{m}}
\newcommand{\Manoa}{M\=anoa}
\newcommand{\Hawaii}{Hawai\kern.05em`\kern.05em\relax i }

\title{Spectral stability and perturbation results for kernel differentiation matrices on the sphere}
\author{T.~Hangelbroek\thanks{Department of Mathematics, University of \Hawaii– \Manoa, 2565 McCarthy Mall,
Honolulu, HI 96822, USA, {\tt hangelbr@math.hawaii.edu}. Research supported by by grant DMS-2010051 from the National Science Foundation.}, C.~Rieger\thanks{Philipps-Universit\"at Marburg, Department of Mathematics and Computer Science,
Hans-Meerwein-Stra\ss{}e 6, 35032 Marburg, 
Germany, {\tt riegerc@mathematik.uni-marburg.de}}, G.~Wright\thanks{Boise State University, 1910 University Drive, 83725, Boise, Idaho, USA,  {\tt gradywright@boisestate.edu}. Research supported by by grants DMS-1952674 and DMS-2309712 from the National Science Foundation.}}
\date{}

\begin{document}

\maketitle

\vspace{-0.2in}
\begin{abstract}
We investigate the spectrum of differentiation matrices for certain operators on the sphere that are generated from collocation at a set of “scattered” points $X$ with positive definite and conditionally positive definite kernels. We focus on cases where these matrices are constructed from collocation using all the points in $X$ and from local subsets of points (or stencils) in $X$. The former case are called global methods (e.g., the Kansa or radial basis function (RBF) pseudospectral method), while the latter are referred to as local methods (e.g., the RBF finite difference (RBF-FD) method).  Both techniques are used extensively for numerically solving certain partial differential equations on spheres, as well other domains. For time-dependent PDEs like the diffusion equation, the spectrum of the differentiation matrices and their stability under perturbations are central to understanding the temporal stability of the underlying numerical schemes. 

In the global case, we present a perturbation estimate for 
differentiation matrices which discretize operators 
that commute with the Laplace-Beltrami operator.
In doing so, we demonstrate that if such an operator has
negative  spectrum, then the differentiation matrix
does, too.
For conditionally positive definite kernels this is particularly challenging since the differentiation matrices are not necessarily diagonalizable.
This perturbation theory is then used to obtain bounds on the spectra of the differentiation matrices that arise from a local method using conditionally positive definite surface spline kernels. Numerical results are presented to confirm the theoretical estimates.
\end{abstract}

{\small \noindent\textbf{Keywords}: Kansa method, RBF Pseudospectral, Hurwitz stability, local Lagrange, RBF-FD}

\medskip

{\small \noindent\textbf{MSC Codes}: 65D12, 65D25, 65M06, 65M20, 65N12}

\section{Introduction}
Kernel-based collocation methods have become increasingly popular for approximating solutions of partial differential equations (PDEs) on spheres, $\Sph^d$, and other smooth surfaces as they do not require a grid or mesh and they can produce high-orders of accuracy (e.g.,~\cite{QTLG,fuselierwright_high,FW,fornberg_flyer_2015,SHANKAR2014JSC,LSW2016,Wendland2020}). These methods are typically formulated in terms of differentiation matrices (DMs) that approximate the underlying continuous spatial differential operators, generically denoted by $\opL$, of the PDE at a set of ‘scattered” nodes (or point cloud) $X$. For time-dependent problems, these kernel-based collocation methods are typically used in a method-of-lines approach, where the spatial derivatives are approximated by DMs and some initial value problem solver is used to advance the semi-discrete system in time~\cite{primerfull}. 

The success of this technique in terms of temporal stability is fundamentally dependent on properties of the spectrum of the DMs.  For example, from classical linear stability analysis, a necessary condition for the method-of-lines scheme to be stable in time is that the spectrum of the DM associated with the spatial derivatives of the PDE (scaled by the time-step) is contained in the stability domain of the initial value problem solver.  At the very least, this condition generally requires that the real part of the eigenvalues of the DM are negative, or non-positive (i.e., the DM is Hurwitz stable). While there are several numerical studies that investigate the spectral properties of kernel-based DMs on spheres and more general surfaces (e.g.,~\cite{fuselierwright_high,LSW2016,FW,FoL11}), there are surprisingly very few theoretical results in the literature despite the increasing popularity of these methods. One reason for this may be that the kernel-based DMs do not immediately inherit any symmetry properties of $\opL$ (e.g., self-adjointness). The aim of this article is to partially fill this gap by developing a spectral stability theory (and an associated perturbation theory) of DMs that arise from discretizations of operators that commute with the Laplace-Beltrami operator $\Delta$. This has application to a wide class of semi-linear parabolic PDEs on spheres, including reaction diffusion equations for modeling pattern formation and chemical signaling~\cite{fuselierwright_high}.

It is worth noting that there are some existing theoretical results on temporal stability of the semi-discrete systems.  For example, the study~\cite{platte2006eigenvalue} shows that DMs for certain operators $\opL$ on $\Sph^2$ have real spectra (a property sometimes called aperiodicity), although this is not sufficient to guarantee Hurwitz stability.  Another study~\cite{QTLG} demonstrates ``energy stability'' for a global collocation method of the heat equation on $\Sph^2$ based on positive definite kernels (this is also a consequence of our results, addressed in Section \ref{SS:energy_stability}, and generalized to the conditionally positive definite case). Finally, some recent works have theoretically studied temporal stability of kernel collocation methods on planar domains using a different approach~\cite{tominec_hyperbolic,Glaubitz}. These studies have primarily focused on hyperbolic PDEs and employ oversampling (or least squares formulations) to demonstrate ``energy stability''.

\subsection{Differentiation matrices}
The global (or Kansa or pseudospectral) kernel-based collocation approach of constructing a DM associated with a given set of distinct points $X\subset{\Sph^d}$ for the differential operator $\opL$ is based on interpolation with a positive definite (PD) or conditionally positive definite (CPD) kernel $\Phi:\Sph^d\times\Sph^d\to \R$ using all points in $X$.  Such DMs can be defined naturally using the Lagrange basis $\{\chi_j\}_{x_j\in X}$
for the trial space $\mathrm{span}_{x_\ell\in X}\Phi(\cdot,x_\ell)$ as
\begin{equation}\label{full_footprint}
\MM_X=  \Bigl(\opL \chi_{j}(x_k)\Bigr)_{j,k}.
\end{equation}
Note that the precise definition of a PD or CPD kernel is given in section 2, and the construction of this DM is described in more detail
in section \ref{S:PD}.

While DMs based on this global method can be computational expensive to compute since they use all the points in $X$, they have theoretical appeal, at least for certain elliptic problems.  For example, in~\cite{EHNRW}, it is shown that, for PD kernels and operators like $\opL=1-\Delta$, 
$\MM_X$ is invertible with a modest stability bound,
see \cite[Proposition 4.4]{EHNRW}, and that the spectrum of $\MM_X$ is strictly positive. One of the primary motivations of this paper is to show that similar results hold for CPD kernels and more general operators $\opL$ involving polynomials of $\Delta$.  For example, for the simpler case of $\opL=-\Delta$, we prove that the spectrum of $\MM_X$ is real and non-negative, as illustrated for a particular $X\subset\Sph^2$ in Figure \ref{fig:example_spectra} ($\times$ markers).

\begin{figure}[th]
\centering
\begin{tabular}{c}
\def\big{\includegraphics[width=0.7\textwidth]{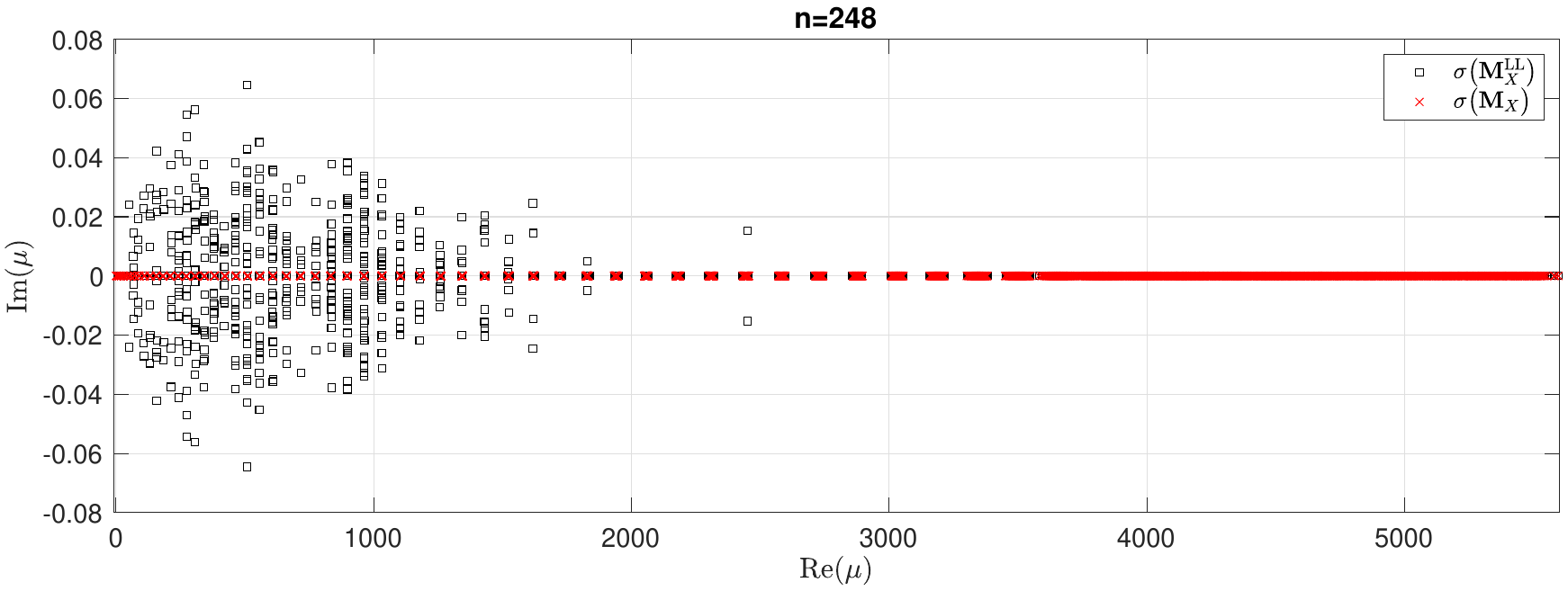}}
\def\little{\includegraphics[width=0.27\textwidth]{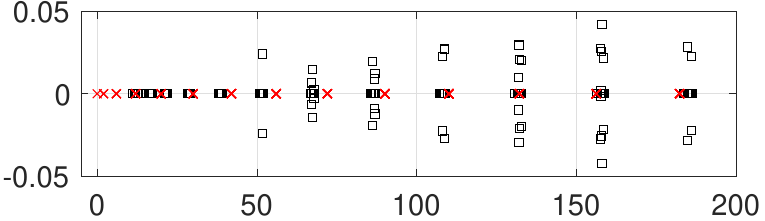}}
\stackinset{l}{40pt}{b}{18pt}{\little}{\big} \\
(a) Local Lagrange \\
\def\big{\includegraphics[width=0.7\textwidth]{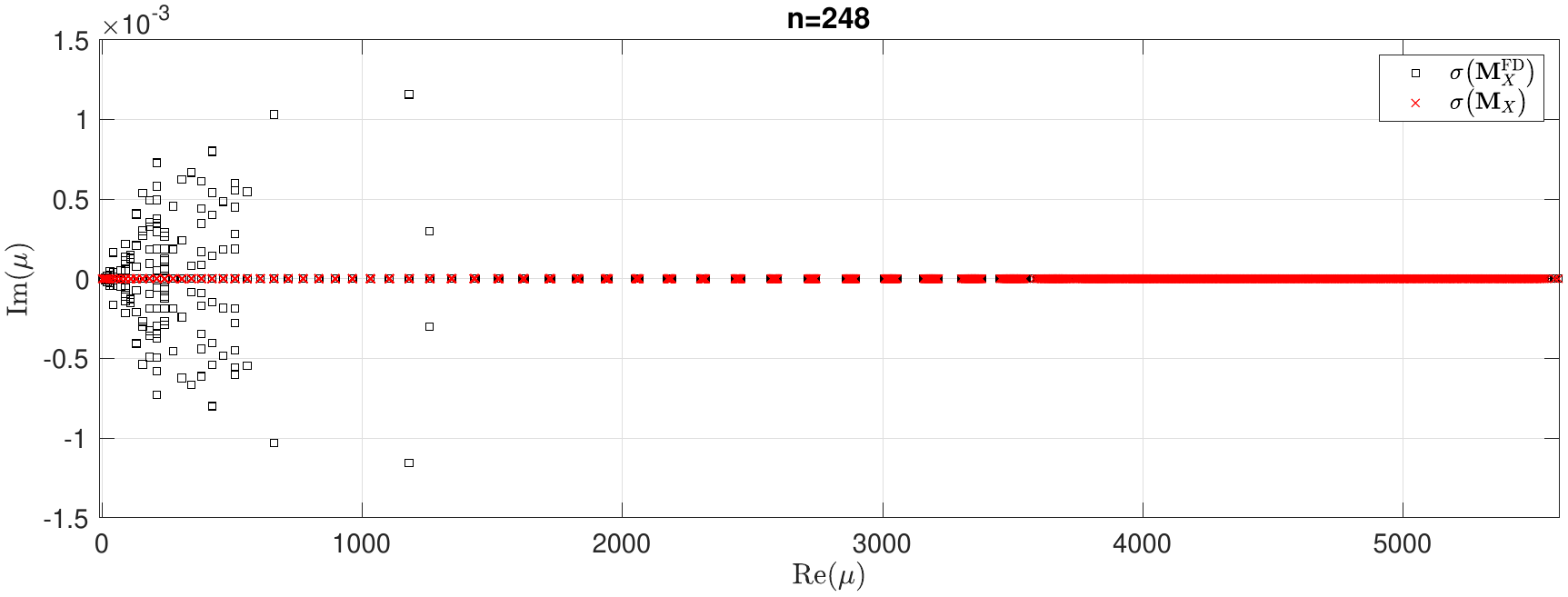}}
\def\little{\includegraphics[width=0.27\textwidth]{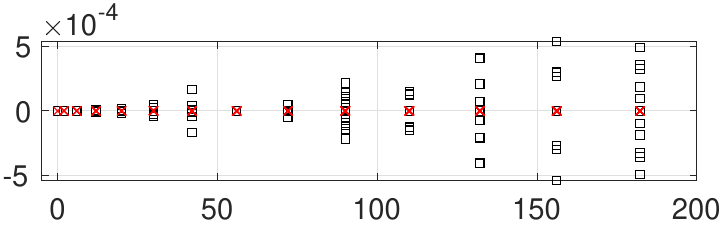}}
\stackinset{l}{40pt}{b}{18pt}{\little}{\big} \\
(b) RBF-FD
\end{tabular}
\caption{Comparison of the spectra of the DMs for the (negative) Laplace-Beltrami operator on $\Sph^2$ based on the global kernel collocation method ($\MM_X$) and two local methods: (a) local Lagrange ($\MLL_X$) and (b) RBF-FD ($\MFD_X$).  All results are for the CPD restricted surface spline kernel of order $m=3$ and a minimum energy point set on $\Sph^2$ (described in section \ref{S:Prelim}) with $N=4096$ points. Both local DMs were computed using stencils with $n=248$ points. Insets show the spectra near the origin of the complex plane. \label{fig:example_spectra}}
\end{figure}

A more computationally efficient approach considers instead constructing a DM by approximating $\opL$ using kernel-collocation over local subsets (or stencils) of $X$.  One of these techniques uses a ``local Lagrange functions'', $\{b_j\}$, generated by a PD or CPD kernel $\Phi$, but only using a small stencil $\Upsilon_j\subset X$ 
consisting of points near to $x_j$, so that $b_j(x_k) = \delta_{j,k}$ for $x_k \in \Upsilon_j$. In this case the DM, which we denote by $\MLL_X$, is sparse and is given by 
 \begin{equation}\label{small_footprint}
 \Bigl( \MLL_X\Bigr)_{j,k}  
 = 
 \begin{cases}  \opL b_{j}(x_k), &x_j \text{ near to }x_k,\\
 0,&\text{otherwise.}
 \end{cases}
 \end{equation}
 For certain point sets $X$ and kernels $\Phi$,  $\opL b_j$  is very close to  $\opL \chi_j$ even for relatively small stencils.
This  is a consequence of results in \cite{FHNWW}, which considers local Lagrange bases for certain CPD kernels on spheres (e.g., restricted surface splines), but  has been generalized to other kernels on other manifolds \cite{HNRW}.
Note that the local Lagrange approach is described in more detail in section \ref{S:Surface_splines}.
Another local technique is the radial basis function finite difference (RBF-FD) method, which has been used widely in many applications on surfaces and general Euclidean domains (e.g.~\cite{fornberg_flyer_2015,SHANKAR2014JSC, OANH, OLEG, tominec,SHANKAR2014JSC}).  This method also constructs a sparse DM, which we denote by $\MFD_X$, but it uses a different stencil-based kernel-collocation approach (as detailed in Section \ref{rbf_fd}). 

While these local methods are much more computationally efficient than the global method, it is more difficult to prove results about the spectral stability of their DMs because they do not have an underlying structure to exploit. Figure \ref{fig:example_spectra} shows that the spectra of DMs of both local methods are indeed more complicated than the global method.  
The approach we follow to study the spectral stability of the local DMs is to view them as perturbations of the the global DM, $\MM_X$, 
and to bound their spectra in terms of the spectrum of $\MM_X$.
  Indeed, both local methods recover $\MM_X$ in the limit that their stencils include every point in $X$. 
This  motivates  two problems:
 \begin{itemize}
 \item to determine the spectrum for the global DM $\MM_X$ when $\Phi$ is CPD
 \item to develop a perturbation theory for the spectrum of  kernel-based DMs
 \end{itemize}
Both of these problems are relevant beyond the specific motivation mentioned here. There are many reasons the DM might be perturbed (e.g. by small adjustments to the point set, or 
evaluation of the kernel, or by approximation of the differential operator, either 
from approximation or by  addition of higher order ``viscosity'' terms),
and it is common practice
to supplement the kernel with an auxiliary polynomial space and provide corresponding 
constraints \cite[Section 5.1]{fornberg_flyer_2015} -- in short, to treat a PD kernel as CPD. 

Our spectral stability analysis is limited to the local Lagrange DMs $\MLL_X$, where, for certain kernels and point sets $X$, recent results allow us to bound $\|\MLL_X - \MM_X\|$.  Analogous results are not yet available for RBF-FD DMs, but we give numerical evidence that a similar stability analysis may hold.

 \subsection{Objectives and Outline}
The main goals of this article are to produce a perturbation result for global kernel-based DMs on spheres for operators involving polynomials of $\Delta$, and
to employ this result to control the spectrum
of nearby local Lagrange DMs. For strictly PD kernels, 
the first goal is achieved by using a fairly straightforward application of the Bauer-Fike theorem.  However, the situation is considerably more challenging for CPD kernels, which produce 
DMs that are not obviously diagonalizable 
(and may have generalized eigenvectors).
However, this is the apposite case for working 
with local Lagrange bases on spheres. 
A final objective is to show that the 
local Lagrange DMs constructed from the restricted surface spline kernels on $\Sph^2$ can have positive spectra.

The outline for the remaining sections of the paper are as follows. 
In section 2, we establish notation and give necessary background for spheres, including the spherical Laplace-Beltrami operator, the operators for which our theory applies, PD and CPD kernels on spheres, and point sets.
Section 3 treats the perturbation of the global DMs
generated by PD kernels. The perturbation result follows by estimating the conditioning of the eigenbasis for the DM and applying the Bauer-Fike theorem. This approach can be generalized to other settings (i.e., other 
choices of kernel/manifold/differential operator) and we discuss how this might done at the end of section 3.

Section  \ref{S:CPD_Stability} introduces global DMs generated 
by CPD kernels.
In section \ref{SS_block_decomp}, we present
Lemma \ref{block_diag_lemma},
which shows that the
DM for a CPD kernel
is similar to a block-upper-triangular matrix
with diagonal principal blocks. 
This factorization is 
sufficient to provide
Proposition \ref{Generalized_BF}, which is
a version of the Bauer-Fike theorem for DMs from CPD kernels.
We then discuss
necessary and sufficient conditions for the DM to be
diagonalizable in section \ref{SS:diagDM}. 
Section \ref{SS_norms_of_block_decomp} provides estimates
of the ($\ell_2$) matrix norms of various blocks appearing in 
the decomposition of Lemma \ref{block_diag_lemma} and 
which appear as terms in Proposition \ref{Generalized_BF}.
Section \ref{SS:R_experiment} gives experimental evidence
which suggests it is possible to improve the estimates
of some of the matrix norms estimated in section 
\ref{SS_norms_of_block_decomp}, and therefore that the 
perturbation result Proposition \ref{Generalized_BF} can be tightened.

Section \ref{S:applications} provides some applications of the theory from Sections 3 and 4 as well as some numerical results. Energy stability of certain semi-discrete problems is given 
in section \ref{SS:energy_stability}. Section \ref{S:Surface_splines} applies the results of section 
\ref{S:CPD_Stability} to the local Lagrange DMs $\MLL_X$ constructed from restricted surface spline kernels on $\Sph^2$.
In this setting, we show that it is possible to consider a relatively sparse $\MLL_X$ as a perturbation of the global DM $\MM_X$ that has a closely matching spectrum. In particular, if the spectrum $\sigma(\MM_X)$ has strictly positive real part, then the spectrum of $\MLL_X$ will also have strictly positive real part for sufficiently large stencils or sufficiently dense point sets $X$.  We illustrate these theoretical results on $\MLL_X$ with some numerical examples.  Finally, we provide some numerical evidence that similar theoretical estimates may also hold for the RBF-FD DMs $\MFD_X$.

\section{Preliminaries}
\label{S:Prelim}
We denote the d-dimensional sphere as $\Sph^{d} = \{x\in \R^{d+1}\mid |x|= 1\}$ and let $\dist(x_1,x_2) := \arccos (x_1\cdot x_2)$ denote the Euclidean distance function for $x_1,x_2\in \Sph^{d}$. 
The sphere has the usual spherical coordinate 
parameterization 
$(\theta_1,\dots, \theta_{d})\mapsto (x_1,\dots,x_{d+1})$
by $\theta_1,\dots,\theta_{d}$ where $0\le \theta_j\le \pi$ for $j=1\dots d-1$
 and $0 \le \theta_{d}\le 2\pi$. Here 
 \begin{eqnarray*}
 x_1&=&\cos \theta_1\\
 x_j &=& \sin\theta_1\cdot \sin \theta_2\cdots \sin \theta_{j-1}\cdot \cos\theta_j 
 \quad
\text{ for }j=2\dots d,\\
 x_{d+1}&=&  \sin\theta_1\cdot \sin \theta_2\cdots \sin \theta_{d-1}\cdot \sin\theta_d.\end{eqnarray*}
The Laplace-Beltrami operator on $\Sph^{d}$ is a self-adjoint, negative semi-definite operator
defined recursively as
$$\Delta_{\Sph^{d}} = 
\frac{1}{\bigl(\sin(\theta_1)\bigr)^{d-1}}
\frac{\partial}{\partial  \theta_1}\bigl(\sin(\theta_1)\bigr)^{d-1} \frac{\partial}{\partial \theta_1}
+(\sin\theta_1)^{-2} \Delta_{\Sph^{d-1}}$$
where $\Delta_{\Sph^{d-1}}$ is the Laplace-Beltrami operator on the $\Sph^{d-1}$ sphere parameterized
by $\theta_2, \dots, \theta_{d}$.

For each $\ell\in \Nats$,
$\nu_\ell:=-\ell(\ell+d-1)$ is an eigenvalue of  the Laplace-Beltrami operator.
The corresponding eigenspace 
has an orthonormal
basis of $N_{\ell}:=
\frac{(2\ell+d-1)\Gamma(\ell+d-1)}{\Gamma(\ell+1)\Gamma(d)}$ 
eigenfunctions,
$\{Y_{\ell}^\mu\}_{\mu=1}^{N_{\ell}}$
{\em spherical harmonics} 
of degree $\ell$.
Thus $$\Delta Y_{\ell}^{\mu} = \nu_{\ell} Y_{\ell}^{\mu}.$$
We denote the space of spherical harmonics of degree
$\ell\le m$ as 
\begin{align}
\label{polyspace}
\Pi_m := 
\mathrm{span}\{Y_{\ell}^{\mu}\mid \mu\le N_{\ell}, \ell\le m\}
\end{align}
and let $M:=\dim{\Pi_m}=\frac{(2m+d)\Gamma(m+d)}{\Gamma(d+1)\Gamma(m+1)}$.  
The collection $(Y_{\ell}^{\mu})_{\ell\ge 0,\mu\le N_\ell}$ forms an 
orthonormal basis for $L_2(\Sph^{d})$.
%

\paragraph{Operator}
Throughout this article we use $\opL$ to denote  
an operator (not  differential, per se)
of the form 
$\opL = p(\Delta)$,
 where $p:\sigma(\Delta)\to \R$ 
 is a function which grows at most algebraically. 
We define $$\lambda_\ell :=p(\nu_\ell),$$ 
and note that $\opL$ is diagonalized by 
spherical harmonics: 
$\opL Y_{\ell}^{\mu}= \lambda_{\ell} Y_{\ell}^{\mu}$.
 
If $p$ is a polynomial of degree $k$,
then  $\opL$ is a differential
operator of order $2k$.
The most elementary choices
are the second order operators  $\opL = \Delta$ generated by 
$p(x)=x$  
and $\opL= 1-\Delta$ generated by $p(x) =1-x$,
but there are many 
higher order operators which arise in practice: $p(x) = x^2$
for the biharmonic, or $p(x) = 1-x^2$ for the linear part
of the Cahn-Hilliard equation as described in \cite{greer2006fourth},
or higher order powers of the Laplace-Beltrami $p(x) = \gamma x^j$
as considered for hyperviscosity terms added to differential operators
in order to stabilize kernel-based methods for certain time-dependent PDEs~\cite{shankar}.

\paragraph{Point sets} 
For $\Omega\subset \Sph^{d}$ and finite subset $X\subset \Omega$, 
we define the {\em fill distance} of $X$ in $\Omega$
as $$h:= \max_{x\in \Omega} \dist(x,X) = \max_{x\in \Omega}\min_{x_j\in X} \dist(x,x_j)$$
and the {\em separation radius} as
$$q:= \frac12\min_{x_j\in X}\min_{x_k\neq x_j}\dist(x_j,x_k).$$
In general, the cardinality $N=\#X$ of $X$ can be estimated above and below by
$c_1 h^{-d} \le N \le c_2 q^{-d}$.

A family of point sets  is called {\em quasi-uniform}
if there is a constant $\rho_0$
so that for every $X$ in the family, the ratio $\rho:= h/q$ is bounded by $\rho_0$. In this setting,
one has $c_1  \rho^{-d} q^{-d} \le N \le c_2 q^{-d}$, and thus $q\sim N^{-1/d}$. Although most of our results hold for general point sets without an a priori bound on the mesh ratio,
some portions (in particular sections 4.4 and 5.1) make use of quasi-uniformity. 

For numerical experiments, we employ some  benchmark families of 
equi-distributed subsets of $\Sph^2$ for our numerical experiments
in sections 4.5 and 5.2; namely, Fibonacci, minimum energy, maximum determinant, and Hammersley point sets
(illustrated in Figure \ref{fig:node}).
The first three of these sets are quasi-uniform, with separation radii $q$ that decay 
with the cardinality $N=\#X$ 
like $N^{-1/2}$, while the Hammersley points are highly unstructured; see~\cite{HMS16} for more information on all of these point sets.

\begin{figure}[h]
\centering
\begin{tabular}{cccc}
\includegraphics[width=0.22\textwidth]{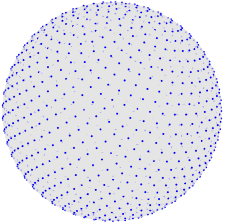} & \includegraphics[width=0.22\textwidth]{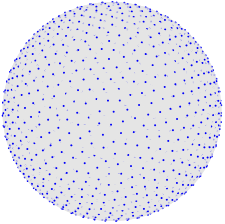} & \includegraphics[width=0.22\textwidth]{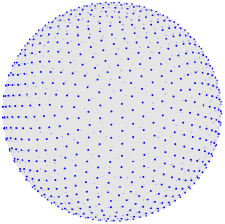} & \includegraphics[width=0.22\textwidth]{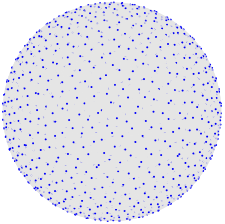} \\
{\small (a) Fibonacci} & {\small (b) Min energy} & {\small (c) Max determinant} & {\small (d) Hammersley}
\end{tabular}
\caption{Examples from the point set families on $\mathbb{S}^2$ used in the numerical experiments. (a) has $N=1025$ points, while for (b)--(d) have $N=1024$ points.\label{fig:node}}
\end{figure}

\paragraph{Conditionally positive definite kernels}
We consider kernels $\Phi:\Sph^{d}\times\Sph^{d}\to \R$ which are 
{\em conditionally positive definite} (CPD) of {\em order} $\m\in\Nats$
(i.e., CPD with respect to the spherical harmonic space $\Pi_{\m-1}$).
This means that for any point set $X\subset\Sph^2$, the {\em collocation matrix}
\begin{align}
\label{collocation_mat}
\PhiB_{X} 
:=\bigl(\Phi(x_j,x_k)\bigr)_{j,k=1\dots N} =
\begin{pmatrix}  \Phi(x_{1},x_{1}) & \dots & \Phi(x_{1},x_{N}) \\
\vdots & \ddots & \vdots  \\
 \Phi(x_{N},x_{1}) & \dots &  \Phi(x_{N},x_{N})
\end{pmatrix} 
\end{align}
 is 
strictly positive definite on the space of vectors 
$$
\left\{\alpha\in \R^N \  
\left\vert \  (\forall p\in \Pi_{\m-1})\, \sum_{j=1}^N \alpha_j p(x_j) 
= 0
\right\}\right..
$$
A kernel is {\em  positive definite} (PD)
if for every $X$, 
the collocation matrix \eqref{collocation_mat} is 
strictly PD on $\R^N$. 
Since such a kernel is CPD with respect to the trivial space
$\Pi_{-1}=\{0\}$, for ease of exposition we use CPD of order $\m=0$ to mean
PD.

For a given point set $X\subset \Sph^d$, define
\begin{align}
\label{trialspace}
S_{X}(\Phi) 
:= \left\{\sum_{j=1}^N a_j \Phi(\cdot, x_j)
 \, \middle|\,
 (\forall p\in \Pi_{\m-1}),\ 
 \sum_{j=1}^N a_j p(x_j)=0\right\}
 +
 \Pi_{\m-1}.
\end{align}

The kernels we consider in this paper are {\em zonal}, which means
that they have a Mercer-like expansion
\begin{equation}\label{HS}
\Phi(x,y) =\sum_{\ell=0}^{\infty} \sum_{\mu\le N_{\ell}} c_{\ell}Y_{\ell}^{\mu}(x) Y_{\ell}^{\mu} (y)
\end{equation}
which is absolutely and uniformly convergent. 
By the addition theorem for spherical harmonics \cite[Theorem 2]{Muller},
$\sum_{\mu\le N_{\ell}} Y_{\ell}^{\mu}(x) Y_{\ell}^{\mu} (y) =\frac{N_{\ell}}{\mathrm{vol}(\Sph^d)  }P_{\ell}(x\cdot y)$,
where  $P_{\ell}$ is an algebraic polynomial of degree $\ell$ 
with $P_{\ell}(1)=1$  and $|P_{\ell}(x)|\le 1$ for all $-1\le x\le 1$ 
by \cite[Lemma 9]{Muller}.\footnote{The functions $P_{\ell}$ are orthogonal with respect to the weight $(1-x^2)^{\frac{d-2}{2}}$, and are thus
 proportional to a Gegenbauer polynomials $C_{\ell}^{(\frac{d-1}{2})}$ as defined in \cite[Chapter 22]{AS}.}
It follows that a series of the form (\ref{HS}) converges uniformly and absolutely on $\Sph^d\times \Sph^d$ if and only if
the coefficients $c_{\ell}$ satisfy
$$\frac{1}{\mathrm{vol}(\Sph^d)}
\sum_{\ell=0}^{\infty} |c_{\ell}| \frac{(2\ell+d-1)\Gamma(\ell+d-1)}{\Gamma(\ell+1)\Gamma(d)}
\sim \sum_{\ell=0}^{\infty}(1+ \ell^{d-1}) |c_{\ell}|<\infty.$$
Here we have used the fact that $N_{\ell} =  \frac{(2\ell+d-1)\Gamma(\ell+d-1)}{\Gamma(\ell+1)\Gamma(d)} \le C_d \ell^{d-1}$
for $\ell\ge 1$ and for a $d$-dependent constant $C_d$.

Note that, by \cite[Theorem 4.6]{menegatto2004conditionally}, 
an expansion of the form (\ref{HS})
yields a  
kernel which is CPD of order $\m$ if and only 
the expansion coefficients  satisfy $c_{\ell}\ge 0$ for all  $\ell\ge \m$, and 
$c_{\ell}>0$ for infinitely many odd and infinitely 
many even values of $\ell$ (the 
PD
case (i.e., $\m=0$) was given in 
\cite[Theorem 3]{chen2003necessary}).
However, most 
CPD
kernels of interest have the property that 
$c_{\ell}>0$ for all $\ell\ge \m$.
We direct the reader to \cite{baxterhubbert} for a number of prominent examples of 
kernels (both PD and CPD)
with precisely calculated coefficients $c_{\ell}$.

\paragraph{Compatibility assumption between $\Phi$ and $\opL$}

For a kernel $\Phi$, and a differential operator $\opL$, 
we adopt the convention that 
 $\opL\Phi(x,y)$ means that the operator is applied to the first argument, i.e.,
 $\opL\Phi(x,y) = \opL^{(1)}\Phi(x,y)$.
 
 The following assumption  guarantees
that $\Phi$ is CPD and
that $\opL\Phi$ is a zonal kernel.
It is in force throughout this paper.
\begin{assumption}\label{PD_assumption}
Let  $\Phi:\Sph^d\times\Sph^d\to \R $ denote 
a  kernel  
with a Mercer-like expansion (\ref{HS}),
which is CPD of order $\m\ge 0 $ (where $\m=0$ means $\Phi$ is PD).
Let $\opL = p(\Delta)$ be a linear operator 
 obtained by  $p:\sigma(\Delta) \to \R$ 
 so that the eigenvalues $\lambda_{\ell}= p(\nu_{\ell})$
 grow at most algebraically, and that
$\sum_{\ell=0}^{\infty} 
c_{\ell} |\lambda_{\ell}|\ell^{d-1}<\infty$.
\end{assumption}

Note that  some of these assumptions have already been mentioned earlier in this section -- the critical 
hypothesis is the compatibility condition between $\Phi$ and $\opL$,
namely  $\sum_{\ell=0}^{\infty} 
c_{\ell} |\lambda_{\ell}|\ell^{d-1}<\infty$.
This condition can be further  simplified if $\opL$ is a differential operator of order $2L$
(in other words, if $p$ is polynomial of degree $L$ or less). 
The requirement that 
$\sum_{\ell=0}^{\infty} c_{\ell} |\lambda_{\ell}| \ell^{d-1}$ 
converges  can be expressed in this case as  
$\sum_{\ell=0}^{\infty}c_{\ell} \ell^{2L +d-1}<\infty$.

\paragraph{Miscellaneous notation}
Throughout the paper, we let $C$ (with or without subscripts) denote a generic positive constant.
To denote matrices, we use bold upper case letters. We also use $\AA\mapsto \|\AA\|$ to denote the induced $\ell_2$ matrix norm.

\section{Spectral stability for PD kernels}
\label{S:PD}

If Assumption \ref{PD_assumption} holds with $\m=0$,
we define the global DM $\MM_{X}$ 
for the
finite point set $X\subset \Sph^d$ 
as the  unique matrix  in $\R^{N\times N}$ which satisfies the identity
$
\MM_{X} \Bigl(
 u{\big|}_X\Bigr)= \bigl( \opL u\bigr){\big|}_X
$
for all $u\in S_X(\Phi)$.
It has  the form 
\begin{equation}\label{DM_PD_def}
\MM_X=\K_X \PhiB_X^{-1},
\end{equation}
where 
$\PhiB_X$ is the collocation matrix \eqref{collocation_mat}
and 
$$\K_X = \Bigl(\opL \Phi(x_j,x_k)\Bigr)_{j,k}.$$ 
It follows that $\K_{X}$ is  a 
symmetric matrix generated by sampling
$
\Psi =\opL \Phi
$.

We can thus conclude that 
the spectrum of $\MM_X$  is the same as that of the symmetric matrix
$\PhiB_X^{-1/2} \K_X \PhiB_X^{-1/2}$.
In particular, if  $\opL$ has non-negative spectrum (each $\lambda_{\ell}\ge 0$),
then $\K_{X}$ is positive semi-definite, so $\sigma(\MM_X)\subset[0,\infty)$.
Furthermore, if
$\lambda_{\ell}c_{\ell}>0$ for infinitely many odd and infinitely many even $\ell$,
then $\K_{X}$ is strictly positive definite, and $\sigma(\MM_X)\subset (0,\infty)$.

%
%
%
\subsection{Conditioning of the eigenbasis of $\MM_X$}
\label{SPD_Case}
While the above setup shows that $\MM_X$ is diagonalizable 
and we have control on $\sigma(\MM_X)$,
the eigenvectors of $\MM_X$  may be far from orthogonal. 
To understand the conditioning of the eigenbasis of $\MM_X$, 
we consider the matrix 
$$\NB_X:=\PhiB_X^{-1/2} \MM_X \PhiB_X^{1/2},$$
which, by the above factorization, satisfies 
$\NB_X =\PhiB_X^{-1/2} \K_X \PhiB_X^{-1/2} $, 
and is therefore symmetric.

\begin{proposition}
\label{PD_perturbation}
Suppose Assumption \ref{PD_assumption} holds with $\m=0$.
If
 $X\subset \Sph^d$ has cardinality $\#X=N$
and  if $\MM^{\epsilon}\in \R^{N\times N}$ 
then for every eigenvalue  $\lambda^{\epsilon}$
of $\MM^{\epsilon}$, 
there is an eigenvalue $\lambda$ of $\MM_X$ with
$$
|\lambda-\lambda^{\epsilon}| 
\le 
\mathrm{cond} (\PhiB_X^{1/2})
\|\MM_X-\MM^{\epsilon}\|
$$
where 
$
\mathrm{cond} (\PhiB_X^{1/2}) = 
\| \PhiB_X^{1/2}\| \|\PhiB_X^{-1/2} \|
$  
is the $\ell_2$ condition number of $\PhiB_X^{1/2}$.
\end{proposition}
\begin{proof}
By the above comment, 
$\NB_X= \UU \DD \UU^{T}$ 
with $\UU$ an orthogonal matrix.
Hence, we have a diagonalization of the form
\begin{eqnarray*}
\MM_X 
= 
\PhiB_X^{1/2} \NB_X \PhiB_X^{-1/2} 
=
 \VV \DD \VV^{-1}
\end{eqnarray*}
with $\VV :=  \PhiB_X^{1/2}  \UU$.
As $\UU$ is orthogonal, 
$\mathrm{cond}(\VV) = \|\VV\|   \| \VV^{-1}\|$
satisfies
$$
\mathrm{cond}(\VV) 
\le 
\| \PhiB_X^{1/2}\|  \|\PhiB_X^{-1/2} \| 
\|\UU\|  \|\UU^{-1} \|
=\mathrm{cond} (\PhiB_X^{1/2}).
$$
The result then follows by an application of
 the Bauer-Fike theorem \cite{BF}.
\end{proof}

In case $\Phi$ has coefficients which have prescribed algebraic decay $c_{\ell} \sim |\nu_{\ell}|^{-m}$,
it is possible to control the smallest eigenvalue of the collocation matrix by a power of 
the separation radius $q$.  
This situation corresponds to a number of kernels with Sobolev native spaces, including
compactly supported kernels and restricted Mat{\'e}rn kernels, as described in \cite[Appendix]{EHNRW}.

\begin{corollary}
\label{PD_cor}
Suppose the hypotheses of Proposition \ref{PD_perturbation} hold, and
the expansion coefficients $c_{\ell}$ of $\Phi$ satisfy
 $ 
 c_{\ell} \sim |\nu_\ell|^{-m}
 $.
 Then
 for every eigenvalue  $\lambda^{\epsilon}$
of $\MM^{\epsilon}$, 
there is an eigenvalue $\lambda$ of $\MM_X$ with
$$
|\lambda-\lambda^{\epsilon}| 
\le 
C q^{-m} 
\|\MM_X-\MM^{\epsilon}\|.
$$
\end{corollary}
\begin{proof}
By \cite[Lemma 4.2]{EHNRW}, 
the decay of the coefficients $c_\ell$ allows us
to bound
$\|\PhiB_X^{-1/2}\|$ 
by $Cq^{d/2-m}$,  while 
$
\|\PhiB_X^{1/2}\| \le \|\PhiB_X\|^{1/2} 
\le Cq^{-d/2}
$ 
since $\Phi$ is bounded and $N \le Cq^{-d}$.
Thus 
$$
\mathrm{cond}(\PhiB_X^{1/2}) \le Cq^{-m}
$$
 and the corollary follows.
\end{proof}
Note that if $ c_{\ell} \sim |\nu_\ell|^{-m}$ and $p$ is a polynomial of degree $L$ or less, 
so $\opL$ is a differential operator of
order $2L$, then
the summability condition $\sum_{\ell=0}^{\infty} c_{\ell}|\lambda_{\ell}|\ell^{d-1}<\infty$ from
Assumption \ref{PD_assumption}
is equivalent to 
$L<m-d/2$.

\subsection{Generalizations to other settings}
The above setup can be generalized considerably to other manifolds 
$\M$ and operators $\opL$.
We begin by assuming that the underlying set $\M$ 
is a metric space with distance function
$\dist:\M\times\M\to [0,\infty)$.

If $\phi_j:\M\to \R$ 
is one of a countable collection of continuous functions whose span is 
dense in $C(\M)$ then any series
\begin{equation}
    \label{HS_M}
\Phi(x,y) = \sum_{j=0}^\infty c_j \phi_j (x) \phi_j(y)
\end{equation}
which converges absolutely and uniformly 
is PD
if $c_j>0$
for every $j$. 
In particular, the collocation matrix 
$\PhiB_X = (\Phi(x_j,x_k))_{j,k}$
is positive definite for any finite set $X\subset \M$.

\bigskip

If $\opL$ is a linear operator diagonalized by the $\phi_j$'s, 
i.e., $\lambda_j \phi_j = \opL \phi_j$ 
and for which the series
$\Psi(x,y) =\sum_j c_j \lambda_j \phi_j (x) \phi_j(y)$
converges absolutely and uniformly, then the DM 
 has the form $\MM_X= \K_X \PhiB_X^{-1}$
where $\K_X =(\Psi(x_j,x_k))_{j,k}$ is symmetric,
and
the perturbation
result of Proposition
\ref{PD_perturbation} holds in this setting. 

\paragraph{Example 1: (intrinsic) compact Riemannian Manifolds}
A common setup involves a  compact, $d$-dimensional  
 Riemannian manifold $(\M,g)$ endowed
 with Laplace-Beltrami operator, given in intrinsic coordinates as
 $$
 \Delta_{\M}
 =
 \frac{1}{\sqrt{\det g(x)}} 
 \sum_{j,k} 
 \frac{\partial}{\partial x_k} 
 \Bigl(\sqrt{\det g(x)} g^{jk}\frac{\partial}{ \partial x_j}\Big).
 $$
 In this case, there exists 
 a countable orthonormal basis  $(\phi_j)$  
 of eigenvectors of $\opL$ 
 with corresponding eigenvalues $\dots \le \nu_1 < \nu_0 = 0$.
We may assume each $\phi_j$ 
to be real-valued.

In this setting, we may again let $\opL = p(\Delta)$ for a function
defined on $\sigma(\Delta_{\M})$.
The requirement of absolute and uniform convergence
of the series $\sum_{j=1}^{\infty} \lambda_j c_j \phi_j(x)\phi_j(y)$
may be simplified by using Weyl laws to control the growth of $|\nu_j|$. 
Estimates in the uniform norm of eigenfunctions can be obtained in specific 
cases, see e.g. \cite[Theorem 2.1]{SOGGE1988} or \cite{donnelly2001bounds}.

There are two  immediate challenges to this approach. 
The first comes purely from adapting  Corollary \ref{PD_cor}.
In order  to estimate the condition number of the eigenbasis 
in terms of 
the separation radius $q$.
As in the case $\M=\Sph^d$, this can be controlled by  estimating the minimal eigenvalue
 of the collocation matrix $\PhiB_X$. 
 To control this eigenvalue  one may attempt to adapt \cite[Lemma 4.2]{EHNRW},
although the much more general spectral theoretic arguments of \cite{GRZ}, specifically
 \cite[Proposition 8]{GRZ} may be more easily adapted to this setting.
 
The second challenge stems form using the kernel $\Phi$ and the DM $\MM_X$.
Although nicely defined by a series, the kernel $\Phi$
may not have a convenient closed-form representation. 
This challenge can be addressed if the eigenfunctions $\phi_j$ appearing in (\ref{HS_M}) are known.
If this is the case, it may be suitable to truncate the Mercer-like series at a sufficiently large threshold $N$.
Of course,  truncating the kernel imposes some error which would need to be managed (for instance
by way of perturbation results like Corollary \ref{PD_cor}).

\paragraph{Example 2: embedded compact Riemannian Manifolds}
If $\M\subset \R^{d+k}$ is a compact Riemannian manifold
embedded in Euclidean space, we may follow \cite{fuselierwright_embedded} 
by considering
a radial basis function (RBF)
 $\varphi:\R^{d+k}\to \R$,
 namely a  function which is symmetric under rotations
(i.e., $\varphi(x) = \tau(|x|)$ for some function $\tau:[0,\infty)\to\R$)
and has Fourier transform which satisfies
$C_1 (1+|\xi|^2)^{-m}\le \widehat{\varphi}(\xi)\le 
C_2 (1+|\xi|^2)^{-m} $
for some $0<C_1\le C_2<\infty$. 

In that case, $\Phi:\M\times \M\to\R:(x,y) \mapsto \phi(x-y)$ is a PD kernel on $\M$. 
By Mercer's theorem, 
it has an absolutely and uniformly convergent
expansion of the form (\ref{HS_M}), where
the functions $\phi_j$ are eigenfunctions of the integral
operator $f\mapsto \int_{\M} f(x) \Phi(\cdot,x)\dif x$. 

The challenge in this case is that
closed form expressions for the eigenfunctions 
$\phi_j$'s and coefficients $c_j$ are not known in general  (see 
\cite{santin2016approximation} for an approach to 
approximate these). 
As a result, it is difficult to identify 
the operators $\opL$
diagonalized by $\phi_j$ which are necessary 
to apply Proposition \ref{PD_perturbation}.

In contrast to the case of Example 1, 
 \cite[Theorem 12.3]{Wend} guarantees that
the collocation matrix $\PhiB_X $ has minimal eigenvalue
$\lambda_{\min}(\PhiB_X) \ge C q^{2m - (d+k)/2}$,
so 
$
\|\PhiB_X^{-1}\|_
\le C q^{(d+k)/2 -2m}$. 
On
the other hand 
$
\|\PhiB_X\|_2
\le
\sqrt{N} \|\Phi\|_{\infty} \le C q^{-d/2}$ 
is guaranteed
by the continuity of $\Phi$ 
and the compactness of $\M$,
so an analogue of Corollary \ref{PD_cor} holds in this case
with estimate
$$|\lambda -\lambda^{\epsilon}|\le Cq^{k/2-m}
\|\MM_X - \MM^{\epsilon}\|.
$$

%
%
%
%
\section{Spectral stability for CPD kernels}
\label{S:CPD_Stability}
We now present our main results for 
global DMs constructed using CPD kernels.
In section \ref{SS_block_decomp} 
we show that $\MM_X$ is similar to 
the sum of a diagonal matrix
and nilpotent matrix 
with a single nonzero $(1,2)$ block. 
This is Lemma \ref{block_diag_lemma}.
In section \ref{SS:gBF} we give a preliminary version 
of the Bauer-Fike theorem which treats the block triangular
factorization.
Section \ref{SS:diagDM} 
considers the full diagonalization of $\MM_X$
under mild assumptions on the operator $\opL$.
We analyze the norms of various factors
appearing in the factorization in Section \ref{SS_norms_of_block_decomp}.

\subsection{DM block decomposition}
\label{SS_block_decomp}
 For a CPD kernel $\Phi$ of order $\m>0$ and operator $\opL$ 
 which satisfy Assumption \ref{PD_assumption},
 and for a  point set $X\subset \Sph^d$,
 define the DM $\MM_{X}$ so that  
 $
 \MM_{X} \Bigl(u\left|_{X}\right.\Bigr) = \bigl( \opL u\bigr)\left|_{X}\right. 
 $
 holds for all $u\in S_{X}(\Phi) $.
 Then $\MM_{X}$ can  be expressed as 
 \begin{equation}
 \label{FD_def}
 \MM_{X} 
 =
 \bigl( \opL \chi_{j}(x_k)\bigr)_{j,k} 
  = 
  \K_{X} \AA 
  + \mathbf{ P}\LLambda \BB,
 \end{equation}
 where
 $
 \K_{X}  = \Bigl(\opL\Phi(x_j,x_k)\Bigr)_{j,k}
 $ 
 and 
 $\mathbf{P} = \bigl(p_{j}(x_k)\bigr)_{j\le M}$ 
 is an $N\times M$ Vandermonde
 matrix associated with a basis $\{p_j\}$ for $\Pi_{\tilde{m}-1}$ (see \eqref{polyspace}). 
 Here we take 
 $p_{j}=  Y_{\ell}^{\mu}$, 
 for a suitable enumeration $(\ell,\mu) \leftrightarrow  j$, %
 so that $\LLambda$ is the diagonal matrix such that
 $ 
 \mathbf{ P}\LLambda =\Bigl( \lambda_{\ell} \, p_j(x_k) \Bigr)_{k,j}
 $.
 The matrices 
 $\AA$
 and $\BB$ occur as solutions 
 to the augmented kernel collocation problem
 \begin{equation}
 \label{aug_coll_problem}
 \begin{pmatrix} 
      \PhiB_{X} &\mathbf{P}\\ \mathbf{P}^T &  \zero
 \end{pmatrix}
 \begin{pmatrix}\AA\\
 \BB\end{pmatrix}
 = 
 \begin{pmatrix}\Id_{N}\\  \zero \end{pmatrix},
 \end{equation}
 where 
 $
 \PhiB_{X} 
 :=
 \Bigl(\Phi(x_j,x_k)\Bigr)_{j,k}
 $ 
 is the kernel collocation matrix. 
 The kernel collocation matrix is a saddle-point matrix and we 
 will exploit this structure and known facts for those matrices 
 as they can be found for instance in \cite{BGL} later on.

Let $\PP^{\dagger}:= (\PP^T\PP)^{-1}\PP^T$ be the standard left inverse 
of $\PP$ and $\QQ = \Id_N-\PP\PP^{\dagger}$ 
be the orthogonal projector onto the space $\bigl(\text{range}(\PP)\bigr)^{\perp} $ where
$$
\bigl(\text{range}(\PP)\bigr)^{\perp} 
= 
\Pi_{\m-1}(X)^{\perp} 
= 
\Bigl\{a\in\R^N\mid 
\sum_{j=1}^N
a_j p(x_j)\; \text{for all }p\in \Pi_{\m-1}(\R^d)\Bigr\}.
$$
Let $\WW\in \R^{N\times(N-M)}$ be a matrix whose columns 
form an orthonormal basis for $\Pi_{\m-1}(X)^{\perp}$.
Then $\WW \WW^T =\QQ$.

By \cite[(3.8)]{BGL}, we can write
\begin{align}
\AA&=\WW(\WW^T \PhiB_X \WW)^{-1}\WW^T \label{A_coeffs}\\
\BB &= \PP^{{\dagger}} (\Id_N - \PhiB_X \AA). \label{B_coeffs}
\end{align}
The first equation shows that $\AA$  is positive semi-definite
and that
\begin{equation}
\label{A_proj_identity}
 \AA =
 \QQ \AA \QQ = 
  \AA \QQ = 
\QQ \AA.
\end{equation}
From  (\ref{B_coeffs}), we can begin to decompose $\MM_X$, obtaining first
\begin{eqnarray}
\MM_X
&=& 
\K_X \AA + \PP\LLambda  \PP^{\dagger}(\Id_N - \PhiB_X \AA)\nonumber\\
&=& 
\PP\LLambda  \PP^{\dagger} 
+ 
(\K_X - \PP \LLambda \PP^{\dagger} \PhiB_X)\AA.
\label{first_block}
\end{eqnarray}
Adding and subtracting
$\PP \PP^{\dagger} \K_X$,
we write the second term in (\ref{first_block}) as
\begin{align*}
 (\K_X &- \PP \LLambda \PP^{\dagger} \PhiB_X)\AA
=
(\K_X  - \PP \PP^{\dagger} \K_X ) \AA 
+ 
\PP
\Bigl( 
    \PP^{\dagger} \K_X-\LLambda  \PP^{\dagger} \PhiB_X
\Bigr)
\AA.
\end{align*}
The first term in this expression is $\QQ\K_X \AA$, 
which equals $\QQ\K_X\AA\QQ$ by (\ref{A_proj_identity}).
This gives the identity:
\begin{equation}
\label{first_decomposition}
\MM_X = \PP\LLambda \PP^{\dagger} 
+
 \QQ\K_X \AA\QQ
+
\PP
\Bigl( 
    \PP^{\dagger} \K_X-\LLambda  \PP^{\dagger} \PhiB_X
\Bigr)
\AA.
\end{equation}

The first term in (\ref{first_decomposition}) 
corresponds to a diagonalizable block. 
Indeed, restricting $\MM_X$ to the invariant subspace $\Pi_{m-1}(X)$
yields   
$\MM_X v = \PP\LLambda  \PP^{\dagger} v$ 
for 
$v\in \Pi_{\m-1}(X)$. 

The second term in (\ref{first_decomposition}) can also be viewed  as a  block of $\MM_X$, 
this time restricted to $\Pi_{\m-1}(X)^{\perp}$ (recall that $\AA$ annihilates $\Pi_{\m-1}(X)$).
Indeed,
$\QQ\K_X \AA\QQ$ 
can be diagonalized, as we did in section \ref{SPD_Case}.
This can be made more explicit with a change of basis using $\WW$.

\paragraph{Change of basis} 
For a matrix $\CC\in \R^{N\times N}$, let 
\begin{equation}\label{COB}
\widehat{\CC} 
:= \WW^T \CC \WW\in
\R^{( N-M)\times (N-M)}
\end{equation}
Recall that the columns of $\WW$ form an 
orthonormal basis for $\Pi_{\m-1}(X)^{\perp}$. Thus 
$\widehat{\PhiB_X} = \WW^T \PhiB_X \WW$ and 
its inverse 
$\widehat{\AA} = \WW^T \AA \WW = (\WW^T \PhiB_X \WW)^{-1}$ 
are both positive definite.

By the same reasoning, $\widehat{\K}_X = \WW^T \K_X \WW$ 
is symmetric, and if $\lambda_{\ell}c_\ell\ge 0$ for $\ell\ge \m$, then 
$\widehat{\K}_X $ is positive semi-definite.
If, in addition,
$\lambda_{\ell}c_{\ell}\neq 0$ for infinitely many even values
of $\ell$ and infinitely many odd values of $\ell$,
then 
 \cite[Theorem 4.6]{menegatto2004conditionally} guarantees that
$\widehat{\K}_X $ is strictly positive definite
(this occurs, for instance, if $\opL=p(\Delta)$ and $p$ is a polynomial).

This yields  the following result.
%
%
\begin{lemma}
\label{block_diag_lemma}
  If Assumption \ref{PD_assumption}
  holds with $\m>0$
and  $X\subset \Sph^{d}$ is finite,
then
the DM
$\MM_X$ has factorization
$$
\MM_X= 
\VV 
\begin{pmatrix} \LLambda &\RR\\ \zero& \TTheta\end{pmatrix} 
\VV^{-1}$$
where $\TTheta\in \R^{(N-M)\times (N-M)}$ 
 and
$\LLambda\in \R^{M\times M}$ are
diagonal, and each  entry
$\LLambda_{j,j} $
is
determined by the spectrum of $\opL$ on $\Pi_{\m-1}$; i.e.,
with 
$\LLambda_{j,j} = \lambda_{\ell}$ where $p_j=Y_{\ell}^{\mu}$.

If $\lambda_{\ell}c_{\ell}\ge 0$ for all $\ell\ge \m$, then each diagonal entry of $\TTheta$ is non-negative, and if, furthermore,
$\lambda_{\ell}c_{\ell}\neq 0$ for infinitely many even and infinitely many odd values of $\ell$, 
then each diagonal entry of
 $\TTheta$ is strictly positive. 
\end{lemma}
\begin{proof}
The second term in (\ref{first_decomposition}) can be written as 
\begin{equation}
\label{cob}
\QQ \K_X \AA \QQ 
= \WW \widehat{\K_X \AA} \WW^T 
= \WW \widehat{\K}_X \widehat{\AA} \WW^T,
\end{equation}
since $ \K_X \AA  =\K_X \QQ \AA$. 
Denote by
$$\SS:=(\widehat{\AA})^{1/2}\in \R^{(N-M)\times (N-M)}$$
the symmetric positive definite square root of $\widehat{\AA}$, 
i.e., $\widehat{\AA} = \SS^T \SS$.  
By symmetry of 
$ \SS \widehat{\K}_X \SS^T $,
it follows that 
$ \SS \widehat{\K}_X \SS^T =  \UU \TTheta \UU^T $
for an orthogonal matrix $\UU\in \R^{(N-M)\times(N-M)}$ 
and a diagonal matrix $\TTheta = \diag(\theta_1,\ldots,\theta_{N-M})$.

Hence, we have the factorization
\begin{equation}
\label{KA_diag}
 \widehat{\K}_X \widehat{\AA} 
= \SS^{-1} \Bigl(\SS \widehat{\K} \SS^T \Bigr) \SS
= \SS^{-1} \UU \TTheta \UU^T \SS.
\end{equation}
Define
\begin{equation}
\label{Z_def}
\ZZ := \WW\SS^{-1} \UU\in \R^{N\times(N-M)}.
\end{equation}
Note that $ \UU^T \SS \WW^T$
is a left inverse of $\ZZ$.
Writing  $\ZZ^{\dagger} := \UU^T \SS \WW^T$,
we see that 
$ \ZZ\ZZ^{\dagger}=\ZZ  (\UU^T \SS \WW^T)=\WW\WW^T= \QQ$.
Combining (\ref{cob}) and (\ref{KA_diag}), we observe 
that 
$\QQ\K_X \AA \QQ
=
  \WW\SS^{-1} \UU\TTheta\UU^T\SS\WW^T
=
\ZZ\TTheta \ZZ^{\dagger}  
$, which yields 
\begin{equation}
\label{second_decomposition}
\MM_X = 
 \PP\LLambda \PP^{\dagger} 
+
\ZZ\TTheta \ZZ^{\dagger}  
+
\PP
\Bigl( 
    \PP^{\dagger} \K_X-\LLambda  \PP^{\dagger} \PhiB_X
\Bigr)
\AA.
\end{equation}
Using $\AA= \AA \QQ = \AA \ZZ\ZZ^{\dagger}$, 
we have
\begin{equation}\label{third_block_diag}
\PP
\Bigl( 
\PP^{\dagger} \K_X-\LLambda \PP^{\dagger}\PhiB_X
\Bigr)
\AA 
= 
\PP \RR \ZZ^{\dagger}
\end{equation}
with 
\begin{equation}\label{R_def}
\RR:= 
\Bigl(  
    \PP^{\dagger} \K_X-\LLambda \PP^{\dagger}\PhiB_X
\Bigr)
\AA \ZZ 
\in \R^{M\times (N-M)}.
\end{equation}
From (\ref{second_decomposition}),  and (\ref{third_block_diag}),
 we have the factorization 
\begin{equation}
\label{Block_decomposition}
\MM_X 
= 
\begin{pmatrix} \PP& \ZZ\end{pmatrix}
\begin{pmatrix} \LLambda &\RR \\ \zero & \TTheta \end{pmatrix}
\begin{pmatrix} 
    \PP^{\dagger}\\ \ZZ^{\dagger} 
\end{pmatrix}
\end{equation}
Finally,
the square matrix matrix $\VV:= \begin{pmatrix} \PP& \ZZ\end{pmatrix}$ satisfies
$$\begin{pmatrix} \PP& \ZZ\end{pmatrix} \begin{pmatrix} \PP^{\dagger}\\ \ZZ^{\dagger}\end{pmatrix} =\PP\PP^{\dagger} +\QQ=\Id_N$$
so $\VV^{-1} = \begin{pmatrix} \PP^{\dagger}& \ZZ^{\dagger}\end{pmatrix}^T$.
By similarity, the spectrum of $\MM_X$ is 
$$\sigma(\MM_X) = \{\lambda_{\ell}\mid \ell<\m\}
\cup 
\{\theta_j\mid j\le N-M\}.
$$ 
Finally, we note that if $\lambda_{\ell}\ge 0$ for all $\ell\ge \m$, then 
$\widehat{\K}_X $ is positive semi-definite, and
so is
$\SS \widehat{\K}_X \SS$.
This implies that 
$\TTheta  = \UU^T\SS \widehat{\K}_X \SS^T  \UU $,
is positive semi-definite.
Since $\TTheta$  is diagonal,
each $\theta_j\ge0$. 
The last statement follows from the observation that $\widehat{\K}_X$ is positive definite in this case.
\end{proof}

\begin{remark}
    As a consequence of Lemma \ref{block_diag_lemma} we obtain that 
    $\sigma(\MM_X) \subset (0,\infty)$ for an operator $\opL=p(\Delta)$ with $p:\sigma(\Delta) \to (0,\infty)$ also for a CPD kernel of order $\tilde{m}$.
\end{remark}
More quantitative results will be shown in Lemma \ref{separation}.
\subsection{A generalization of the Bauer-Fike Theorem}
\label{SS:gBF}
By adapting the argument from \cite{Chu},
we obtain the following estimate on perturbation of eigenvalues of $\MM_X$.
\begin{proposition}
\label{Generalized_BF}
Under conditions of Lemma \ref{block_diag_lemma},
if $\MM^{\epsilon}
\in \R^{N\times N}$
and
if 
$\mu\in \sigma(\MM^{\epsilon})$,
then
$$
\dist(\mu,\sigma(\MM_X))
\le 
\max\left(
2 \kappa \|\MM_X -\MM^{\epsilon}\|,
\sqrt{2\kappa \|\RR\|  
\|\MM_X - \MM^{\epsilon}\|}
\right)
$$
holds with 
$\kappa:= 
\mathrm{cond}(\VV) = 
\|\VV\| \|\VV^{-1}\|
$. 
We note that $\RR$ and $\VV$ are matrices appearing in the
decomposition of $\MM_X$ given in 
Lemma \ref{block_diag_lemma}.
\end{proposition}
\begin{proof}
Assume without loss that
$\mu\in \sigma(\MM^{\epsilon})\setminus \sigma(\MM_X)$.
For the invertible matrix $\VV$  given in (\ref{Block_decomposition}),
and for $\EE :=\MM_X - \MM^{\epsilon}$, 
we have the factorization 
\begin{eqnarray*}
\VV^{-1}(\mu \Id - \MM^{\epsilon} )\VV 
&=& 
 \mu\Id 
 - 
 \begin{pmatrix} \LLambda&\RR\\0&\TTheta\end{pmatrix} 
 + 
 \VV^{-1}\EE \VV\\
&=& 
\left( 
    \mu\Id 
    - 
    \begin{pmatrix} \LLambda&\RR\\0&\TTheta\end{pmatrix}
\right)
\Bigl(\Id + \tilde{\EE}\Bigr),
\end{eqnarray*}
where
$
\tilde{\EE}
:= 
\left( 
    \mu\Id - \begin{pmatrix} \LLambda&\RR\\0&\TTheta\end{pmatrix}
\right)^{-1}
\VV^{-1} \EE \VV
$.
Since
$\VV^{-1}(\mu\Id - \MM^{\epsilon})\VV$  is singular and 
$\mu\notin \sigma(\MM_X)$,
it follows from the above factorization that 
$\Id + \tilde{\EE}$ is singular, 
and therefore we have that 
$\| \tilde{\EE}\|\ge 1$.
This ensures that 
$$
\left\|
    \left( 
        \mu\Id - 
        \begin{pmatrix}\LLambda&\RR\\0&\TTheta\end{pmatrix}
    \right)^{-1} 
\right\|^{-1} 
\le 
\|\VV^{-1} \EE \VV\|
\le  
\kappa \|\MM_X - \MM^{\epsilon}\|.
$$
In other words,
\begin{eqnarray}
\left\|  
    \begin{pmatrix} 
        \mu\Id_{M}-\LLambda&-\RR\\
        0&\mu\Id_{N-M}-\TTheta
    \end{pmatrix}^{-1}
\right\|^{-1}
&\le& 
\kappa(\VV) \|\MM_X - \MM^{\epsilon}\|.
\label{lower_bound}
\end{eqnarray}
At this point, we may  control the left hand side 
of (\ref{lower_bound}) from below by  estimating 
\begin{multline*}
F(\mu):=
\left\|  
    \begin{pmatrix} 
     \mu\Id_{M}-\LLambda&-\RR\\
     0&\mu\Id_{N-M}-\TTheta
    \end{pmatrix}^{-1}
\right\|
\\
=
\left\|  
  \begin{pmatrix} 
  (\mu\Id_{M}-\LLambda)^{-1}&
  (\mu\Id_{M}-\LLambda)^{-1}\RR(\mu\Id_{N-M}-\TTheta)^{-1}\\
  0&
  (\mu\Id_{N-M}-\TTheta)^{-1}
  \end{pmatrix}
\right\| .
\end{multline*}
By a  triangle inequality:
\begin{eqnarray*}
F(\mu)
&\le&
\left\|  
  \begin{pmatrix} 
  (\mu\Id_{M}-\LLambda)^{-1}&0\\
  0&(\mu\Id_{N-M}-\TTheta)^{-1}
  \end{pmatrix}
\right\| \\
&& 
+
\left\|  
  (\mu\Id_{M}-\LLambda)^{-1}\RR(\mu\Id_{N-M}-\TTheta)^{-1}
\right\| 
\end{eqnarray*}

Both
 $|(\mu-\TTheta)_{j,j}^{-1}|$  and 
 $|(\mu-\LLambda)_{k,k}^{-1}|$ 
 can be controlled by  $1/\dist(\mu,\sigma(\MM_X)) $
 for any $j\le N-M$ and any $k\le M$.
 Thus,
 $$F(\mu)
 \le 
 \bigl(\dist(\mu,\sigma(\MM_X))\bigr)^{-1}  
 +
 \bigl( \dist(\mu,\sigma(\MM_X)) \bigr)^{-2} 
 \|\RR\| 
 .
 $$

 \paragraph{Case 1}
 If 
 $ \bigl(\dist(\mu,\sigma(\MM_X))\bigr)^{-1}  
 \le
 \bigl( \dist(\mu,\sigma(\MM_X)) \bigr)^{-2}  \|\RR\| $,
 then the upper bound on $F(\mu)$ becomes
 $F(\mu)
\le 2  \bigl( \dist(\mu,\sigma(\MM_X)) \bigr)^{-2} 
 \|\RR\|$. 
 So  (\ref{lower_bound}) guarantees that 
 \begin{equation}\label{case1}
 \frac{\bigl( \dist(\mu,\sigma(\MM_X)) \bigr)^{2} }{2\|\RR\|}
 \le  
 \kappa  \|\MM_X - \MM^{\epsilon}\|.
 \end{equation}

 \paragraph{Case 2}
 If 
 $ \bigl(\dist(\mu,\sigma(\MM_X))\bigr)^{-1}  
 >
 \bigl( \dist(\mu,\sigma(\MM_X)) \bigr)^{-2}  \|\RR\| $,
 then the upper bound on $F(\mu)$ implies that
 $F(\mu)
\le 2  \bigl( \dist(\mu,\sigma(\MM_X)) \bigr)^{-1} 
$. 
 So  from (\ref{lower_bound})
 we have
 \begin{equation}
 \label{case2}
 \frac{\bigl( \dist(\mu,\sigma(\MM_X)) \bigr) }{2}
 \le  
 \kappa \|\MM_X - \MM^{\epsilon}\|.
 \end{equation}
 \smallskip
 The result follows from (\ref{case1}) and (\ref{case2}).
 \end{proof}

%
%
%
\subsection{Diagonalizing the DM}
\label{SS:diagDM}
From Lemma \ref{block_diag_lemma}
 it follows that $\MM_X$ is diagonalizable if and only if 
$\begin{pmatrix} \LLambda &\RR \\ \zero & \TTheta \end{pmatrix}$ is diagonalizable.
We can recast this by way of the Sylvester problem: 
find $\XX\in\R^{M\times (N-M)}$ so that 
\begin{equation}
\label{Sylv}
-\LLambda \XX + \XX \TTheta = \RR.
\end{equation}
If (\ref{Sylv}) has a solution, then
\begin{equation}
 \label{diagonalizing}
\begin{pmatrix} \LLambda &\RR \\ \zero & \TTheta \end{pmatrix} = 
\begin{pmatrix} \Id_M &\XX \\ \zero & \Id_{N-M} \end{pmatrix}
\begin{pmatrix} \LLambda &\zero \\ \zero & \TTheta \end{pmatrix}
\begin{pmatrix} \Id_M &-\XX \\ \zero & \Id_{N-M} \end{pmatrix},
\end{equation}
so the block upper triangular matrix is diagonalizable.
There are practical solution methods for problems of the form (\ref{Sylv}) 
under significantly more general conditions than we use 
(just the assumption that the spectra of 
$\sigma(\LLambda)$ and $\sigma(\TTheta)$ 
are separated, see \cite{bickley1960matrix,bartels1972solution} 
and \cite{Higham} for an overview). 
In our case, where $\LLambda$ and $\TTheta$ are diagonal,  the 
solution is very simple.
We may rewrite (\ref{Sylv}) with the help of the matrix
$\GGamma = \Bigl(  \theta_j-\lambda_i \Bigr)_{i,j}$. 
In that case, (\ref{Sylv}) has the form
$\GGamma \odot \XX = \RR$, where $\odot$ is entry-wise multiplication.
This leads to three cases:
\begin{enumerate}
\item 
The spectra of $\LLambda$ and $\TTheta$ are disjoint, 
in which case $\GGamma$ has no zero entries, and $\XX$
has a unique solution obtained by entry-wise division.
\item 
There is some overlap between spectra of $\LLambda$ and $\TTheta$, 
but each zero entry of $\GGamma$ corresponds
to a zero entries of $\RR$. 
In this case, there are many solutions to (\ref{Sylv}). 
\item $\TTheta$ and $\LLambda $, 
have overlapping spectra and 
$\RR_{i,j}\neq 0$ for some $i,j$ for which $\GGamma_{j,j} = 0$.
In  this case, $\LLambda_{i,i} =\TTheta_{j,j}$ 
has a generalized eigenvector.
\end{enumerate}
If either case 1.\ or 2.\ holds, then (\ref{Sylv}) has a solution

Under basic hypotheses on the operator $\opL$, we can estimate the separation between $\sigma(\LLambda)$ and $\sigma(\TTheta)$, and therefore
we can estimate the norm of $\XX$.
%
%
%
\begin{lemma}
\label{separation}
Suppose Assumption \ref{PD_assumption} holds.
Define numbers $\lambda_{\flat},\lambda^{\sharp}\in \R$ 
  as
 $$
\lambda_{\flat} := \max\{ p(\nu_{\ell}) \mid \ell< \m\}
\quad
\text{ and }
\quad
\lambda^{\sharp} := \min\{p(\nu_{\ell})\mid \ell\ge \m\}.$$
If $\lambda_{\flat}<\lambda^{\sharp}$, 
then $\min_{j\le N-M}\TTheta_{j,j}-\max_{i\le M}\LLambda_{i,i} \ge \lambda^{\sharp}- \lambda_{\flat}>0$.
\end{lemma}
%
%
%
\begin{proof}
Set
 $\tilde{\opL}
 := \opL - \lambda^{\sharp}
 \mathrm{Id}$.
 Then 
 $\tilde{\Psi}= \tilde{\opL}\Phi$ 
 is a zonal kernel and, 
 moreover,
 is CPD of order $\m$.
 By definition,
$\Psi-\tilde{\Psi} = \lambda^{\sharp} \Phi$.

For $X\subset \Sph^d$, 
let $\K_X := (\Psi(x_j,x_k))_{j,k}$ and
$\tilde{\K}_X := (\tilde{\Psi}(x_j,x_k))_{j,k}$.
Both 
matrices are symmetric,
and we 
 have that $\tilde{\K}_X = \K_X -  \lambda^{\sharp}\PhiB_X$.
Define $\tilde{\LLambda} :=\LLambda - \lambda^{\sharp} \Id_M$, 
and note that $\tilde{\LLambda}$
satisfies the property that 
$\PP \tilde{\LLambda } = \PP \LLambda - \lambda^{\sharp}\PP$.
Thus 
the DM for
$\tilde{\opL}$ is
$$\tilde{\MM}_X %
= \tilde{\K}_X \AA + \PP \tilde{\LLambda } \BB
= \MM_X - \lambda^{\sharp} \Id_N.$$ 
From Lemma \ref{block_diag_lemma}, it follows  that $\tilde{\MM}_X$ has spectrum 
 $$
 \sigma(\tilde{\MM}_X) = \{p(\nu_{\ell} ) - \lambda^{\sharp}\mid \ell<\m\}\cup \{\tilde{\theta}_j\mid j\le N-M\},$$ 
 and that the first component satisfies the inclusion
  $$\{\lambda_{\ell}  - \lambda^{\sharp}\mid \ell<\m\}\subset (-\infty, \lambda_{\flat} - \lambda^{\sharp}]\subset (-\infty,0)$$
  while $\{\tilde{\theta}_j\mid j\le N-M\}\subset [0,\infty)$.
  Thus, for $\MM_X$, the DM for $\opL$, 
  we have the disjoint union
  $$\sigma(\MM_X) = \sigma(\tilde{\MM}_X) +\lambda^{\sharp}= \{\lambda_{\ell}   \mid \ell<\m\}  \sqcup  
  \{\theta_j\mid j\le N-M\}.$$
  \end{proof}
  
In  case the hypotheses of Lemma \ref{separation} hold, 
we can diagonalize $\MM_X$ and make use of the 
standard Bauer-Fike theorem instead of Proposition \ref{Generalized_BF}.

\begin{theorem}
\label{Generalized_BF_Diagonalized}
Suppose Assumption \ref{PD_assumption} holds with $\m>0$.
If the {\em separation} 
$$ \gamma:=\min_{\ell\ge \m}\lambda_{\ell}
 -
 \max_{\ell<\m}\lambda_{\ell}
 $$ is positive, 
then for any $N$-set $X\subset \Sph^d$,
$\MM_X$ is diagonalizable, and for 
matrix $\MM^{\epsilon}\in\R^{N\times N}$ and 
eigenvalue
$\mu\in \sigma(\MM^{\epsilon})$ we  have
$$\dist(\mu,\sigma(\MM_X))\le  
\Bigl(1+\frac{\|\RR\|}{\gamma}\Bigr)^2
\|\VV\| 
\|\VV^{-1}\| 
\|\MM_X-\MM^{\epsilon}\|.$$
Here $\RR$ and $\VV$ are matrices appearing in the
decomposition of $\MM_X$ given in 
Lemma \ref{block_diag_lemma}.
\end{theorem}
\begin{proof}
From Lemma \ref{block_diag_lemma}
and (\ref{diagonalizing}),
$\MM_X$ has the  factorization:
 \begin{equation*}
\MM_X = \tilde{\VV} \begin{pmatrix} \LLambda &\zero \\ \zero & \TTheta \end{pmatrix}
\tilde{\VV}^{-1}
\qquad\text{where}\quad
\tilde{\VV}:=\VV\begin{pmatrix} \Id_M &\XX \\ \zero & \Id_{N-M} \end{pmatrix}
 \end{equation*}
 and where $\GGamma \odot \XX = \RR$.
The condition number for the eigenbasis is controlled by
  $\kappa(\tilde{\VV}) \le (1+\|\XX\|)^2\|\VV\|\|\VV^{-1}\|$.
 We can estimate $\|\XX\| \le \frac{1}{\gamma}\|\RR\|$.
 The result then follows from the standard Bauer-Fike theorem.
  \end{proof}
%
%
%


%
%
%
\subsection{Matrix norms of elements of the block decomposition}
\label{SS_norms_of_block_decomp}
We now restrict the situation to 
kernels  which are CPD of order $\m$ 
and for which 
the coefficients in the expansion (\ref{HS}) satisfy
$c_\ell \sim |\nu_{\ell}|^{-m}$ for all $\ell\ge\m$.

\paragraph{Estimating $\|\PP\|$ 
and $\|\PP^{\dagger}\|$}
\label{SSS:Vandermonde}
We can estimate $\|\PP\|$
via H{\"o}lder's inequality. 
$$
\sum_{k\le M}|\sum_{j\le N} a_j p_j(x_k)|^2
\le 
\sum_{k\le M} \|a\|_2^2 \sum_{j\le N} |p_j(x_k)|^2
\le 
N M \max_{j\leq M} \|p_j\|_{\infty}^2 \|a\|_2^2,
$$
so 
$\|\PP\|
\le 
\sqrt{NM}\max_{j\le M}\|p_j\|_{\infty}$. 
Since $M$ is assumed to be fixed and $N\le C q^{-d}$, 
we can express this estimate as
\begin{equation}
    \label{general_P_norm}
\|\PP\|\le C q^{-d/2}.
\end{equation}
Naturally, the 
bound (\ref{general_P_norm})
holds for $\|\PP^T\|$ as well.

A simple consequence of
\cite[Theorem 4.2]{NPW-L} guarantees that 
the Gram matrix
$\PP^T \PP$ has spectrum contained in the interval
$[C_1 h^{-d}, C_2q^{-d}]$ 
for constants $C_1$ and $C_2$ independent of $X$
(see \cite[Lemma 4.2]{EHNRW} 
for this argument in the case $d=2$). 
It follows that 
$\|(\PP^T \PP)^{-1}\| \le C h^{d}\le C \rho^d q^d$, 
where $\rho= h/q$ is the {\em mesh ratio} of $X$,
and $C$ is a constant that only depends on $\m$ and $d$. 
Thus, 
\begin{equation}\label{dagger}
\|\PP^{\dagger}\|
\le 
\|(\PP^T \PP)^{-1}\| \|\PP^T \| \le C \rho^d q^{d/2}.
\end{equation}

\paragraph{Estimating   $\| \ZZ\|$ 
and $\|\ZZ^{\dagger}\|$\label{SSS:Z}}

By  \cite[Proposition 5.2]{FHNWW}, the minimal eigenvalue $\lambda_{min}$
of $\WW^T \PhiB_X \WW$ 
is bounded below by $\lambda_{min} \ge C_1 q^{2m-d}$
for some constant $C_1>0$. 

Additionally, since $\SS$ is the positive square root of $\widehat{\AA}$ and 
$$\|\SS\| = \|\widehat{\AA}^{1/2}\|  = \|\widehat{\AA}\|^{1/2} = \|(\WW^T \PhiB \WW)^{-1}\|^{1/2} \leq C q^{d/2-m}$$ 
we have, since $\ZZ^{\dagger} = \UU^T \SS \WW^T$, that
$$
\|\ZZ^{\dagger}\| = \| \SS \|\le C q^{d/2-m}.
$$ 
From (\ref{Z_def}) we know $\|\ZZ\| =\|\SS^{-1}\|$, so it follows that 
$$\|\ZZ\|
=
\|\SS^{-1}\|= \|\widehat{\AA}^{-1/2}\|
=
\|\WW^T \PhiB_X \WW\|^{1/2}
\le \|\PhiB_X\|^{1/2}.$$
By the estimate $\|\PhiB_X\|\le \|\Phi\|_{\infty} N$,
we have 
$$
\|\ZZ\|
\le C q^{-d/2}.
$$

\paragraph{Estimating   $\| \RR\|$}
Treating perturbations of the kernel DM in  
the CPD 
case requires handling the upper right hand block
$\RR$ in  the decomposition (\ref{Block_decomposition}).

To estimate 
$\|\RR\| = 
\Bigl\| \Bigl(  
\PP^{\dagger}
\K_X-
\LLambda 
\PP^{\dagger}
\PhiB_X\Bigr)\AA \ZZ\Bigr\|$,
note  that
\begin{itemize}
\item  $\|\ZZ\|\le C q^{-d/2}$;
\item $\|\AA\| = \|\widehat{\AA}\| \le Cq^{d-2m} $ by (\ref{A_coeffs});
\item $\|\PP^{\dagger} \K_X\|\le \|\PP^{\dagger}\| \|\K_X\|
\le C\rho^d q^{-d/2}$; 
\item $\|\LLambda \PP^{\dagger} \PhiB_X\|\le (\max_{\ell<\m} \lambda_{\ell} )\|\PP^{\dagger}\| \|\PhiB_X\|
\le C\rho^d q^{-d/2}$.
\end{itemize}
The latter two  use the matrix norms 
$\|\K_X\|$ and $\|\PhiB_X\|$, which can be
estimated by introducing  supremum norms 
$\|\Phi\|_{\infty}= \max_{x,y}|\Phi(x,y)|$ 
and $\|\opL \Phi\|_{\infty} =\max_{x,y}|\opL\Phi(x,y)|$, 
to obtain
$\|\PhiB_X\| \le N \|\Phi\|_{\infty}$ 
and $\|\K_X\| \le N \|\opL\Phi\|_{\infty}$.
Recalling that $N\le C q^{-d}$,
we obtain the bound 
\begin{equation}
    \label{general_R}
\|\RR\|
\le C \rho^d q^{-d/2}  q^{d-2m} q^{-d/2}
\le C \rho^d
q^{-2m} .
\end{equation}

It is conceivable that a much better estimate is possible, since
the  factor 
$\PP^{\dagger}
\K_X-
\LLambda 
\PP^{\dagger}
\PhiB_X$,
which involves a commutator-like factor 
(namely $\PP^{\dagger}\opL - \opL \PP^{\dagger}$),
has been roughly estimated with a triangle inequality.
We investigate this numerically in the next section.

\paragraph{Estimating the condition number of $\VV$}

 We now consider the condition number 
 $\kappa =\mathrm{cond}(\VV)= \|\VV\| \|\VV^{-1}\|$
 appearing in Lemma \ref{Generalized_BF}.
 
 The blocks $\PP$ and $\ZZ$ of 
 $\VV = \begin{pmatrix}\PP &\ZZ\end{pmatrix}$
  have orthogonal ranges and nullspaces.
 Indeed, we have 
 $$\Ran \PP = \Pi_{\m-1}(X) =\Nul \QQ= \Nul \WW^{T}=\Nul \ZZ^{\dagger} $$
and
$$\Ran \ZZ =\bigl(\Pi_{\m-1}(X)\bigr)^{\perp}= \Ran \WW= \Ran \QQ=\Nul \PP^{\dagger}.$$
 Thus, $\|\VV\|
 \le \max(\|\PP\| ,\|\ZZ\|)$
 and
 $\|\VV^{-1}\| 
\le \max(\|\PP^{\dagger}\| ,\|\ZZ^{\dagger}\|)$.
It follows that 
 $$
\|\VV\|
 \le \max(\|\PP\| ,\|\ZZ\|) \le Cq^{-d/2}
 $$
 If $c_\ell \sim |\nu_{\ell}|^{-m}$, 
 we have 
$$
\|\VV^{-1}\| 
\le \max(\|\PP^{\dagger}\| ,\|\ZZ^{\dagger}\|) 
\le 
Cq^{d/2} \max(\rho^d,q^{-m})
.
$$
Together, this implies
  $$\kappa\le C \max(\rho^d,q^{-m}) .$$

  \begin{corollary}\label{CPD_cor}
Under hypotheses of Proposition \ref{Generalized_BF}, if
we assume the kernel's expansion coefficients
$c_{\ell}$ satisfy 
$c_\ell \sim |\nu_{\ell}|^{-m}$ for all $\ell\ge \m$
and $\rho\lesssim q^{-m/d}$,
then for a matrix 
$\MM^{\epsilon}$ 
sufficiently close to $\MM_X$,
and 
$\mu\in \sigma(\MM^{\epsilon})$ 
there is a $\lambda \in \sigma(\MM_X)$ for which
\begin{eqnarray*}
    |\mu-\lambda|&\le &
     C 
 \max
 \left(q^{-m/2}\|\RR\|^{1/2}\sqrt{\|\MM -\MM^{\epsilon}\|},
 q^{-m}\|\MM -\MM^{\epsilon}\|\right)\\
 &\le&
 C 
 q^{-3m/2}\sqrt{\|\MM_X -\MM^{\epsilon}\|}.
\end{eqnarray*}
  \end{corollary}
  \begin{proof}
 Plugging the estimate $\kappa\le Cq^{-m}$
 into  Proposition \ref{Generalized_BF} gives 
 the first inequality.
 Using  the bound (\ref{general_R})
 for $\|\RR\|$
 gives
  the second. 
\end{proof}
%
 \subsection{Numerical estimates on $\|\RR\|$}  \label{SS:R_experiment}
In this section, we give numerical evidence that much better estimates for $\|\RR\|$ than \eqref{general_R} may be possible.  
We consider DMs for $\opL=-\Delta$ on $\Sph^2$ using the four families of point sets discussed in section \ref{S:Prelim} and illustrated in Figure \ref{fig:node}. 

\begin{figure}[htb]
\centering
\begin{tabular}{cc}
\includegraphics[width=0.45\textwidth]{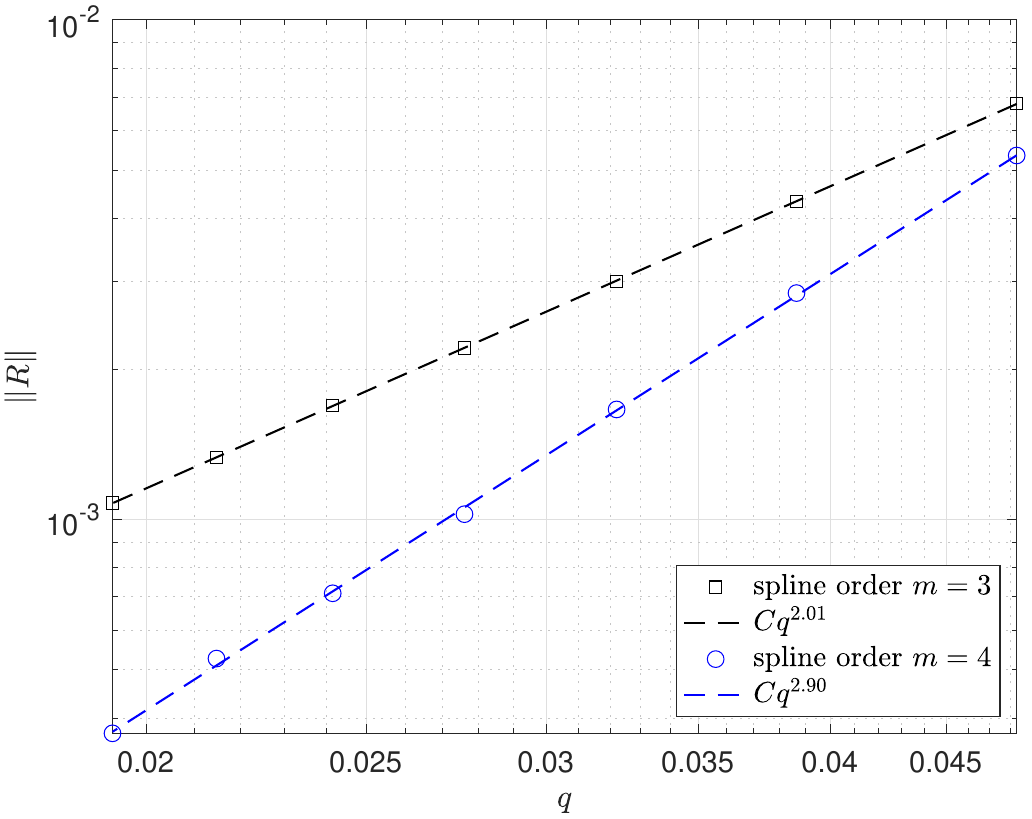} & \includegraphics[width=0.45\textwidth]{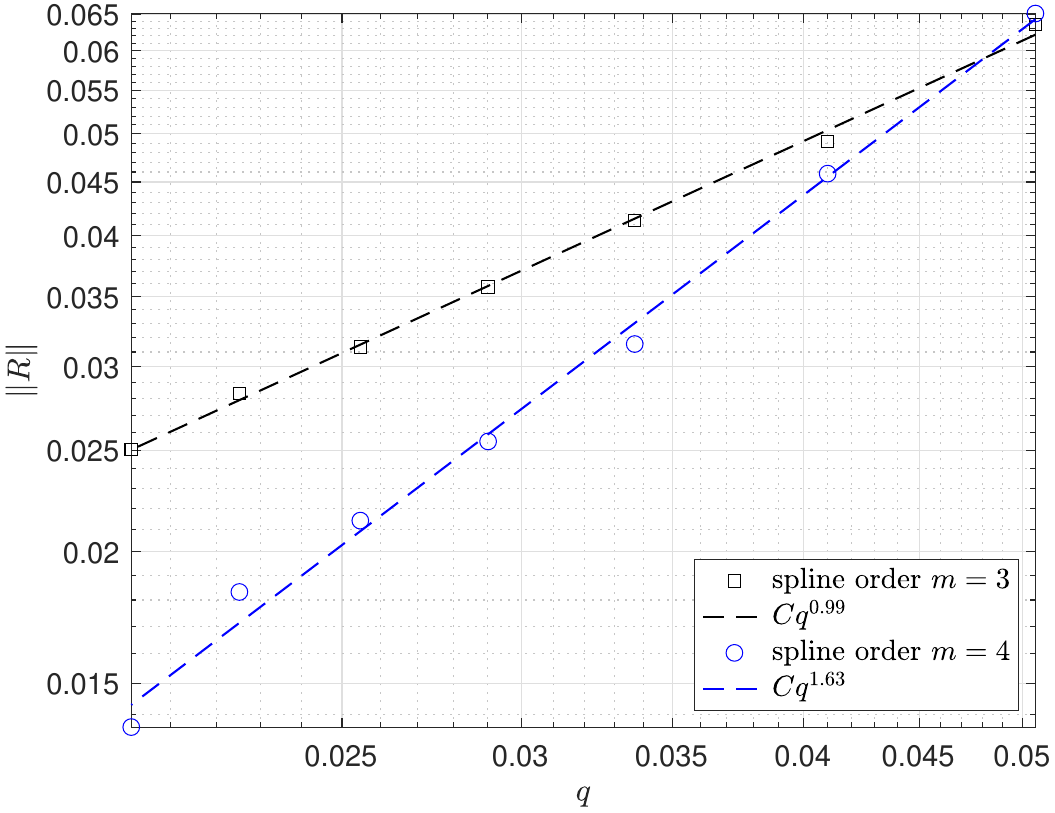} \\
(a) Fibonacci & (c) Maximum determinant \\
\includegraphics[width=0.45\textwidth]{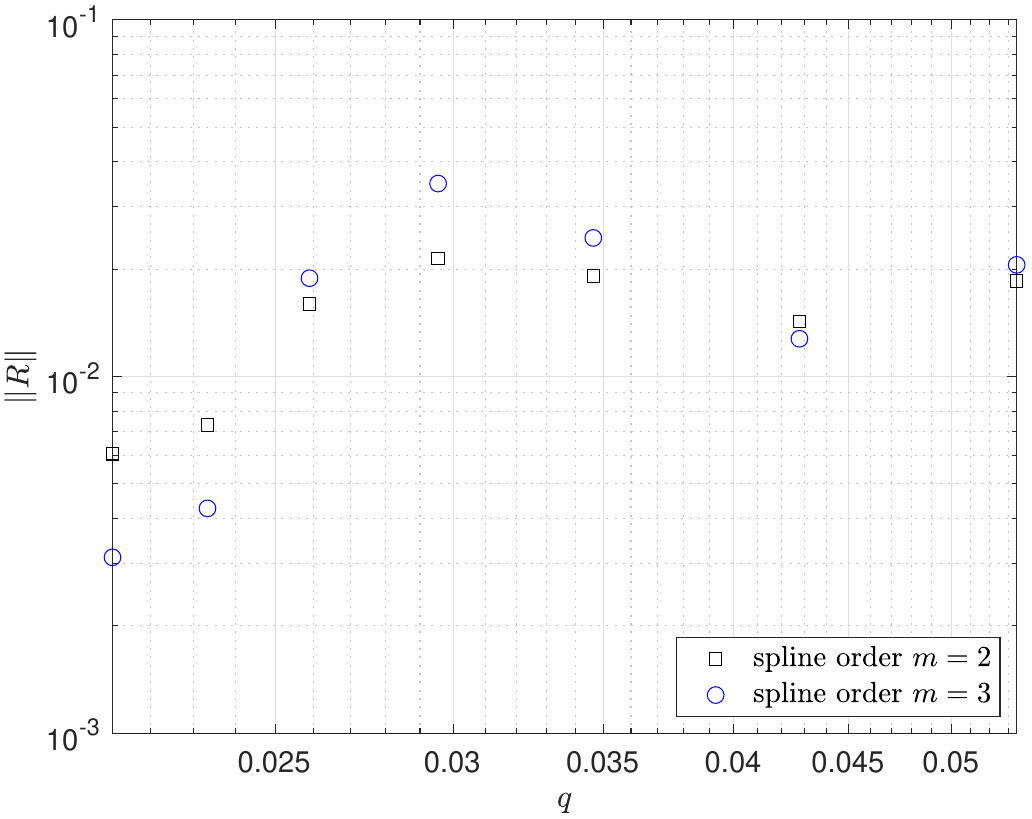} & \includegraphics[width=0.45\textwidth]{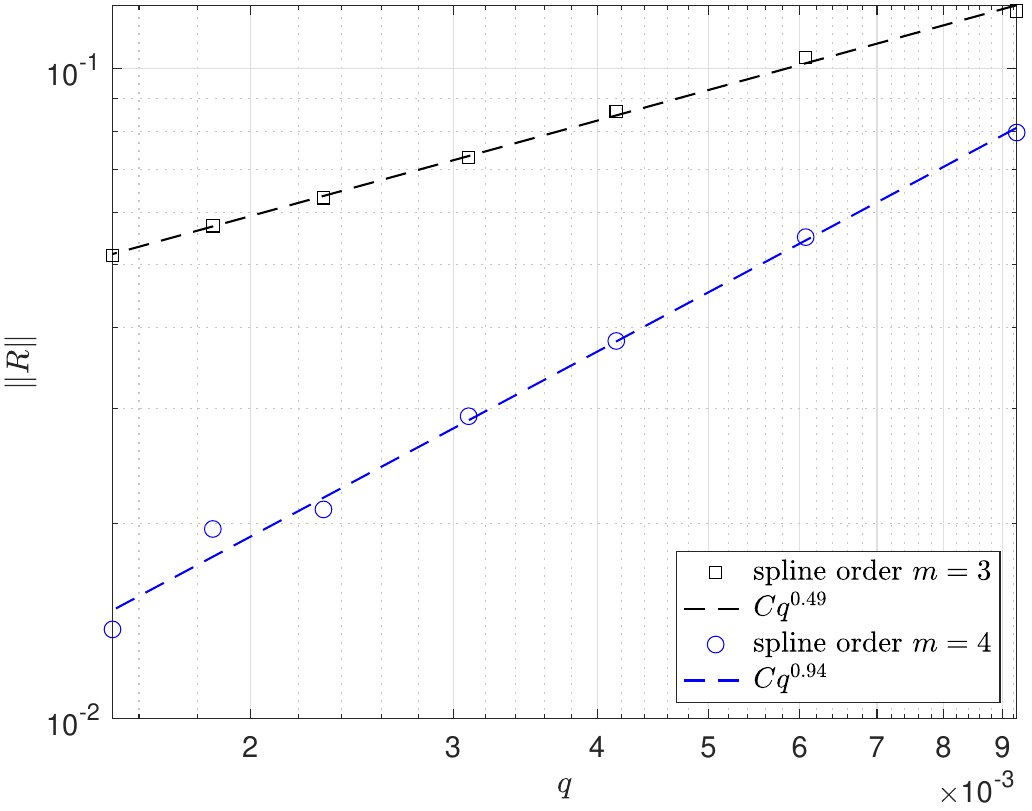} \\
(b) Minimum Energy & (d) Hammersley
\end{tabular}
\caption{Numerical results on $\|\RR\|$ vs.\ the separation radius $q$ for different point set families on $\mathbb{S}^2$ using $\opL = -\Delta$.  Each plot shows the results for the restricted surface \eqref{surface_spline} spline kernels of order $m=3$ and $m=4$ using augmented spherical harmonic spaces $\Pi_{m-1}$.  The dashed lines in (a), (c), \& (d) show the lines of best fit (on a log-scale) to the data, which indicate an algebraic decay rate of $\|\RR\|$ with decreasing $q$. The results in (b) do not show a discernible pattern of $\|\RR\|$ in terms of  $q$, so the estimated rates are omitted.\label{fig:NormR}}
\end{figure}

We first consider DMs formed from the restricted surface spline kernels:
\begin{equation}
\label{surface_spline}
    \Phi_m(x,y) = C_m(1-x\cdot y)^{m-1} \log(1-x\cdot y).
\end{equation}
These kernels are CPD of order $\m$, where $\m\geq m$, and have
a Mercer-like expansion (\ref{HS}) with 
coefficients that decay like 
$c_{\ell} \sim |\nu_{\ell}|^{-m}$ for $\ell\ge m$.
Indeed, for $\ell \ge m$, the kernel has expansion (\ref{HS})
with coefficients satisfying
$c_{\ell} = C\prod_{j=0}^{m-1} (\nu_{\ell} +j(j+1))^{-1}$
by \cite[Lemma 3.4]{HSphere}.  We consider the $m=3$ and $m=4$ kernels with augmented spherical harmonic spaces $\Pi_{m-1}$ (the minimum degree space required for well-posedness).  Figure \ref{fig:NormR} displays the results of $\|\RR\|$ computed for points $X$ of increasing cardinality $N$ (and hence decreasing $q$) from each family of point sets. Included in the plots from (a), (c), \& (d) of this figure are the estimated algebraic rates of decay of $\|\RR\|$ in terms of decreasing $q$; these estimates were omitted from (b) since no discernible pattern was evident. We note that the estimated rates in the three figures all involve positive powers of $q$ rather than negative powers as in the bound \eqref{general_R}. Furthermore, even the results in (b) do not follow \eqref{general_R}. 

\begin{figure}[htb]
\centering
\begin{tabular}{cc}
\includegraphics[width=0.42\textwidth]{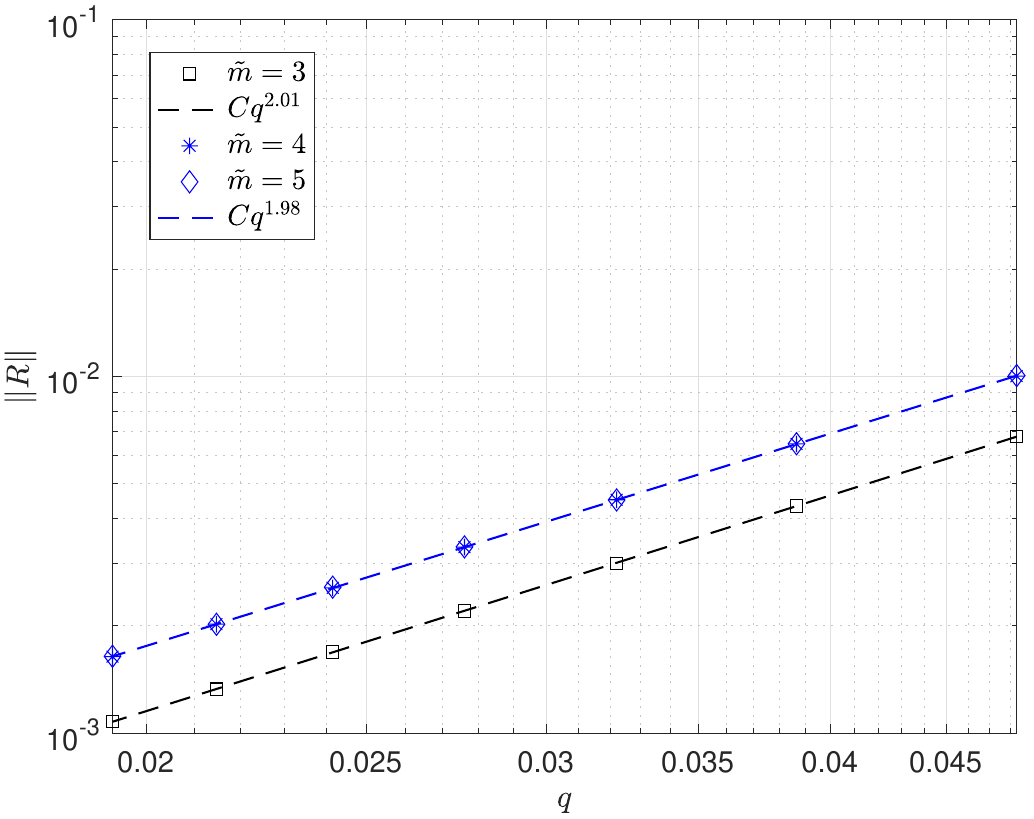} & \includegraphics[width=0.42\textwidth]{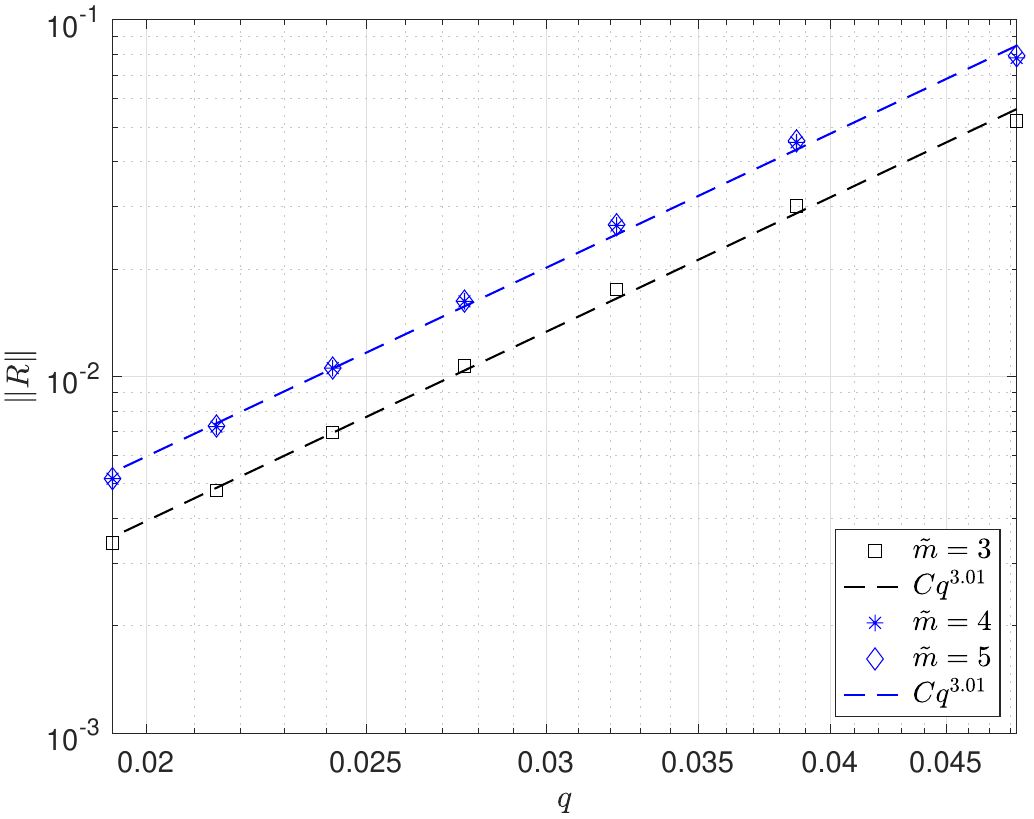} \\
(a) surface spline $m=3$ & (b) inverse multiquadric 
\end{tabular}
\caption{Numerical results on $\|\RR\|$ vs.\ the separation radius $q$ for the Fibonacci points on $\mathbb{S}^2$ using $\opL = \Delta$.  (a) Results from three experiments with the order of the restricted surface spline kernel \eqref{surface_spline} fixed at $m=3$ and the degree of augmented spherical harmonic spaces $\Pi_{\tilde{m}-1}$ changing. (b) Same as (a), but for the restricted inverse multiquadric kernel.  Dashed lines estimate the algebraic decay rate of $\|\RR\|$ with decreasing $q$.\label{fig:NormRSphDeg}}
\end{figure}

In the next experiment we test how $\|\RR\|$ depends on the degree of the augmented spherical harmonic basis $\Pi_{\tilde{m}-1}$, $\tilde{m}\geq m$, when the kernel is fixed. This is especially common for applications of kernel-based methods on $\Sph^2$ (and more general domains) where the order of the spline kernel is kept low and the degree of the polynomial basis is allowed to grow~\cite{shankar,bayona2017role}.  Figure \ref{fig:NormRSphDeg} (a) shows the results for the Fibonacci nodes.  We see from this plot that increasing the degree does not change the algebraic rate of decay of $\|\RR\|$ with decreasing $q$, but only (possibly) the constant.  We note that similar results were observed for the maximum determinant and Hammersley points and are thus omitted.

Finally, we consider the behavior of $\|\RR\|$ for the restricted inverse multiquadric kernel $\Phi(x,y) = (1+\varepsilon^2(1-x\cdot y))^{-1/2}$, where $\varepsilon >0$ is the free shape parameter.  This kernel is PD with coefficients $c_{\ell}$ in \eqref{HS} that decay exponentially fast with $\ell$~\cite{baxterhubbert}. Hence, the estimates for bounding $\|\RR\|$ in \eqref{general_R} do not apply.  Figure \ref{fig:NormRSphDeg} (b) displays the results associated with this kernel using the Fibonacci nodes.  Similar to part (a) we include results for $\|\RR\|$ when including different augmented spherical harmonic bases $\Pi_{\tilde{m}-1}$ with this kernel. As with the restricted surface spline results, $\|\RR\|$ seems to have an algebraic decay rate (of approximately 3) with decreasing $q$ and this rate does not seem to depend on $\tilde{m}$.  We note that similar results were observed (with different algebraic rates) for the other point set families and hence are omitted. 

These experiments suggest that the estimate in \eqref{general_R} is indeed overly pessimistic and better bounds for $\|\RR\|$ may be possible, even ones that extend to kernels with exponentially decaying Mercer expansions. Unfortunately, the geometry of the points seems to affect the algebraic decay of $\|\RR\|$ with decreasing $q$, so any tighter estimates may need to take the geometry into account.

\section{Applications}
In this section consider two applications of the theory of sections \ref{S:PD} and \ref{S:CPD_Stability}.
\label{S:applications}
\subsection{Continuous time-stability for the global DMs}
\label{SS:energy_stability}
Consider the semi-discrete approximation of the equation $\frac{\partial}{\partial t}u = \opL u$ at a set of points $X\subset\Sph^d$, where $\opL=p(\Delta)$ and $\sigma(\opL)\subset(-\infty,0]$ (e.g., for $\opL =\Delta$ corresponding to the diffusion equation).  We denote the approximation as 
\begin{align}
\frac{d}{d t}u_X = \MM_X u_X, 
\label{eq:semi_discrete} 
\end{align}
where $u_X: [0,\infty)\times X\mapsto \R$ and $\MM_X$ is the global DM associated with a PD or CPD kernel $\Phi$. 
 
A straightforward application of the results from sections \ref{S:PD} and \ref{S:CPD_Stability} can be used to show the solution of the semi-discrete system is energy stable, and thus the norm of $u_X(t)$ does not grow in time. The proofs differ depending on the kernel $\Phi$ used to construct $\MM_X$.

\paragraph{PD Kernel}
Using the norm $\|u\|_{\PhiB_X^{-1}}^2 := u^T \PhiB_X^{-1} u$,
we note that 
$$\frac{d}{dt}\|u_X\|_{\PhiB_X^{-1}}^2 =u_X^T (\MM_X^T\PhiB_X^{-1} + \PhiB_X^{-1} \MM_X)u_X
$$
which implies, using (\ref{DM_PD_def}) that
$
\frac12 \frac{d}{dt}\|u_X\|_{\PhiB_X^{-1}}^2 =u_X^T  \PhiB_X^{-1/2} \K_X \PhiB_X^{-1/2} u_X<0$,
so the semi-discrete problem \eqref{eq:semi_discrete} is energy stable.
\paragraph{CPD Kernel} 
In this case, we employ the 
semi-norm $[u]_{\AA}^2 := u^T \AA u$
where $\AA$ is the positive semi-definite matrix given in \eqref{A_coeffs}.
Then
$$\frac{d}{dt}[u_X]_{\AA}^2 =u_X^T (\MM_X^T\AA+ \AA \MM_X)u_X$$
which implies, since $\AA \PP=0$, that
$$\frac12 \frac{d}{dt}[u_X]_{\AA}^2 =u_X^T  \AA \K_X \AA u_X
=(\WW^T u_X)^T  \widehat{ \AA}
\widehat{ \K}_X \widehat{\AA} (\WW^T u_X).
$$
Here we have used (\ref{A_proj_identity}) and the change of basis (\ref{COB}).
Since $\widehat{ \K}_X$ has negative spectrum,
we have 
$\frac{d}{dt}[u_X]_{\AA}^2 <0$
as long as 
$
\WW^{T}
u_X\neq 0$. However, if 
$
\WW^{T}
u_X=0$, then $u_X\in \Pi_{\m-1}(X)$, 
in which case $\MM_X $ is exact and we can use the standard $\ell_2$ norm to show $\frac{d}{dt}\|u_X\|^2 \leq 0$. Thus, the semi-discrete problem \eqref{eq:semi_discrete} is energy stable also for the CPD case.

\subsection{Spectral stability of the local Lagrange DMs from restricted surface splines}
\label{S:Surface_splines}

We apply the results of section \ref{S:CPD_Stability} to
the restricted surface spline kernels 
$\Phi_m$  on $\Sph^2$ defined in
(\ref{surface_spline}).
Recall that $\Phi_m$ is CPD of order $\m$ as long as $\m\ge m$.

When $\m=m$, these kernels have the property, introduced  in \cite{FHNWW}, 
that for quasi-uniform point sets $X$,
the kernel spaces 
$S_X(\Phi_m)$
possess a
 {\em localized basis}:
 a basis $(b_j)_{j\le N}$, which enjoys the following two properties
 (among others)
\begin{itemize}
\item each function $b_j$ employs a small stencil:
$b_j\in S_{\Upsilon_j}(\Phi)$
where  $\Upsilon_j\subset X$ consists of points near to $x_j$,
\item each $b_j$
is close to $\chi_j$ in a variety of norms (roughly, for
the norm of any Banach space in which the native space is embedded
-- this will be made more precise below).
\end{itemize}
The sphere, along with Euclidean space, provides the most readily 
available kernels having localized bases, although they exist in other
settings as well. For any general compact, closed Riemannian manifold,
there exist PD kernels with this property, as shown in \cite{HNRW}, 
although they generally do not have convenient closed form expressions, 
or even expansions of the form (\ref{HS}) in a familiar orthonormal set.
For compact, rank 1 symmetric spaces 
(which include spheres of all dimensions, 
as well as a number of other sphere-like manifolds), 
there exist kernels with Mercer-like expansions 
in the Laplacian eigenbasis, as shown in \cite{HNW-p}, 
for which the ideas of \cite{FHNWW} can be applied to 
construct localized bases -- in particular, this is possible for 
restricted surface splines and similar kernels 
on higher dimensional spheres.

 \paragraph{Perturbation of the DM via localized bases}

Let $X\subset \Sph^2$ have mesh ratio $\rho:= h/q$.
For $K>0$, and for $x_j\in X$,
we define the stencil $\Upsilon_j:= X\cap B(x_j,Kh |\log h|)\subset X$,
and note that $\Upsilon_j$ has cardinality
$\#\Upsilon_j \sim K^2\rho^2  (\log h)^2
=\mathcal{O}\bigl(K^2(\log N)^2\bigr)$.

We define  
 the local Lagrange function 
 $b_j \in S_{\Upsilon_j}(\Phi_m)$ 
via the condition
$$ b_j(x_k) 
=\delta_{j,k},
\qquad \text{ for all }\qquad
x_k\in \Upsilon_j.$$
Note that $b_j\in S_X(\Phi_m)$,
since $\Upsilon_j\subset X$.
By \cite[Lemma 5.2]{EHNRW}, there exist positive constants $\alpha$ and $\beta$ so that
 for any stencil parameter $K>0$
 and
 for sufficiently dense $X\subset \Sph^2$,
 $$\|\opL b_j -\opL  \chi_{j} \|_{L_{\infty}(\Sph^2)} \le C h^J
 \qquad
\text{ holds with 
$J =\alpha K +\beta$.} 
$$
The local Lagrange DM, $\MLL_X$, is constructed by applying $\opL$ to each $b_j$:
$$
\Bigl(\MLL_X\Bigr)_{j,k} =\begin{cases} \opL b_k(x_j), & |x_j-x_k|\le K h|\log h|,\\ 0,& \text{otherwise}. \end{cases}
$$
Since $X$ has mesh-ratio $\rho$ 
 the stencil $\Upsilon_j$ has cardinality
$\#\Upsilon_j 
\le C \rho^2 \bigl(K^2(\log N)^2\bigr)$.
Thus 
there are $\mathcal{O}\bigl(K^2(\log N)^2\bigr)$ nonzero
entries in each row (or column) of $\MLL_X$.
Following the arguments in \cite{EHNRW}, 
specifically estimates \cite[(6.3)]{EHNRW} and \cite[(5.9)]{EHNRW}, 
we have that 
\begin{equation}
\label{perturbation_bound}
   \|\MM_X-\MLL_X\|_2 \le C h^{J-2}.
\end{equation}
We note that the approximation
{\em order} $J$ given by (\ref{perturbation_bound})  
has an linear dependence on the stencil parameter $K$,
namely $J=\overline{\alpha} K + \overline{\beta}$ for some $\overline{\alpha},\overline{\beta}\in\R$, with $\overline{\alpha}>0$.

We can measure the distance between spectra of $\MM_X$ and 
$\MLL_X$ by using either Proposition 
\ref{Generalized_BF} in the most general setting,
or Theorem \ref{Generalized_BF_Diagonalized} 
in cases where $\opL$ separates $\Pi_{m-1}$ from its orthogonal complement.

 \paragraph{General (non-diagonalized) case}
 We apply Proposition  \ref{Generalized_BF}, 
 with $\MM^{\epsilon} = \MLL_X$, and observe 
 that for every eigenvalue $\mu \in \sigma(\MLL_X)$
 there is 
 $\mu^{\star} \in \sigma(\MM_X)$
 for which 
 $|\mu-\mu^{\star}|
 \le 
 C \sqrt{\kappa\|\RR\|} h^{J/2 -1} $.
 At this point, we use the estimates on $\|\VV\|$,
 $\|\VV^{-1}\|$ and $\|\RR\|$ 
 collected in Corollary \ref{CPD_cor} to note that
 $$
 \max_{\mu\in\sigma(\MLL_X)}
\dist\bigl(\mu, \sigma(\MM_X)\bigr)
 \le C  h^{J/2 -1-3m/2}.$$
This suggests that $J$ (and therefore $K$) 
should be chosen
larger than $2m+2$ in order to ensure fidelity 
to the original 
(positive) spectrum of $\MM_X$.

\begin{remark} \label{Remark_pessimistic}
If the pessimistic bound $\|\RR\|\le Ch^{-3m}$ is replaced
 by $\|\RR\| \le C$, 
 as suggested by the experimental results 
 of section \ref{SS:R_experiment},
 then Corollary \ref{CPD_cor} gives the improved estimate
 $$
\max_{\mu\in\sigma(\MLL_X)}
\dist\bigl(\mu, \sigma(\MM_X)\bigr)
 \le C  h^{J/2 -1-m/2}.$$
\end{remark}

\paragraph{Diagonalizable case}
If $\opL$ satisfies
 $\gamma
 =\min_{\ell\ge \m}p(\nu_{\ell}) -\max_{\ell<\m} p(\nu_{\ell})>0$, 
 then we may apply 
 Theorem \ref{Generalized_BF_Diagonalized}
 to obtain 
  $$ \max_{\mu\in\sigma(\MLL_X)}
\dist\bigl(\mu, \sigma(\MM_X)\bigr)\le C  h^{J -2-5m}.$$
In particular, this holds if $\opL = -\Delta$.
(We recall that in this section
we require $\m=m$.)

\begin{remark} \label{Remark_pessimistic_diag}
Here, as in Remark \ref{Remark_pessimistic}, 
if
$\|\RR\|\le Ch^{-3m}$ is replaced
 by $\|\RR\| \le C$, 
 as suggested by the experimental results 
 of section \ref{SS:R_experiment},
 then Theorem \ref{Generalized_BF_Diagonalized} gives the improved estimate
 $$
\max_{\mu\in\sigma(\MLL_X)}
\dist\bigl(\mu, \sigma(\MM_X)\bigr)
 \le C  h^{J -2-m}.$$
\end{remark}

%

%
%
%
\subsubsection{Numerical results on spectra of the local Lagrange DMs}
\begin{figure}[htb]
\centering
\begin{tabular}{cc}
\includegraphics[width=0.47\textwidth]{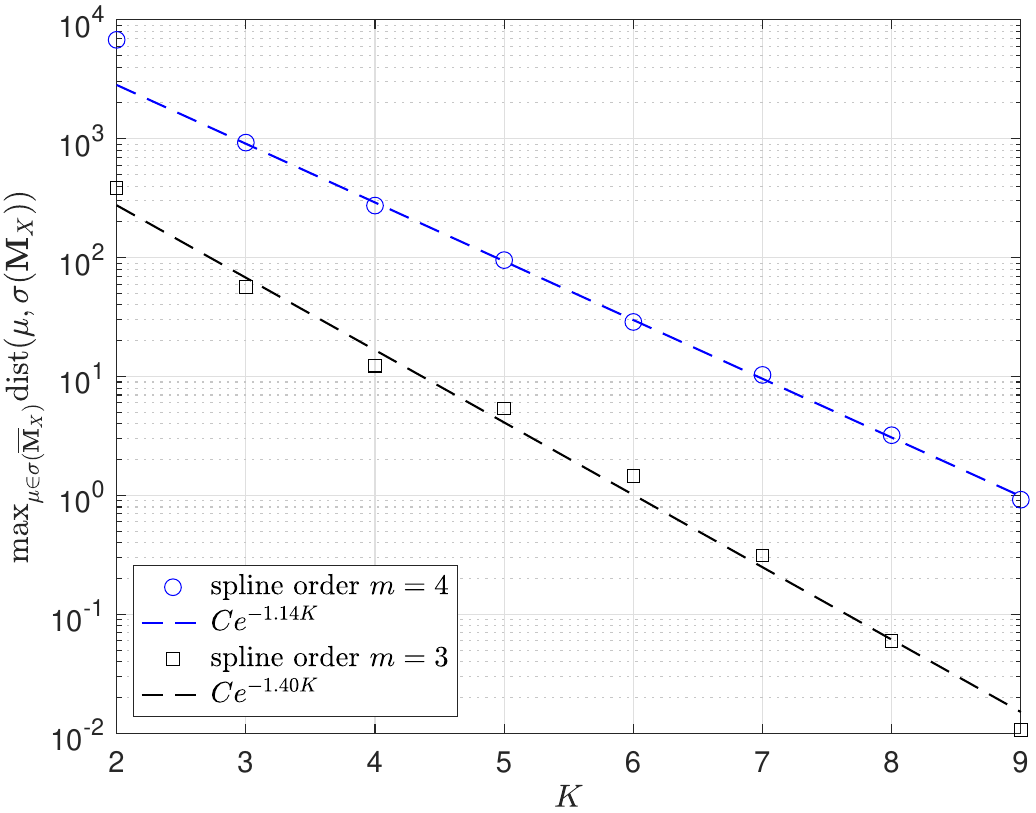} & \includegraphics[width=0.47\textwidth]{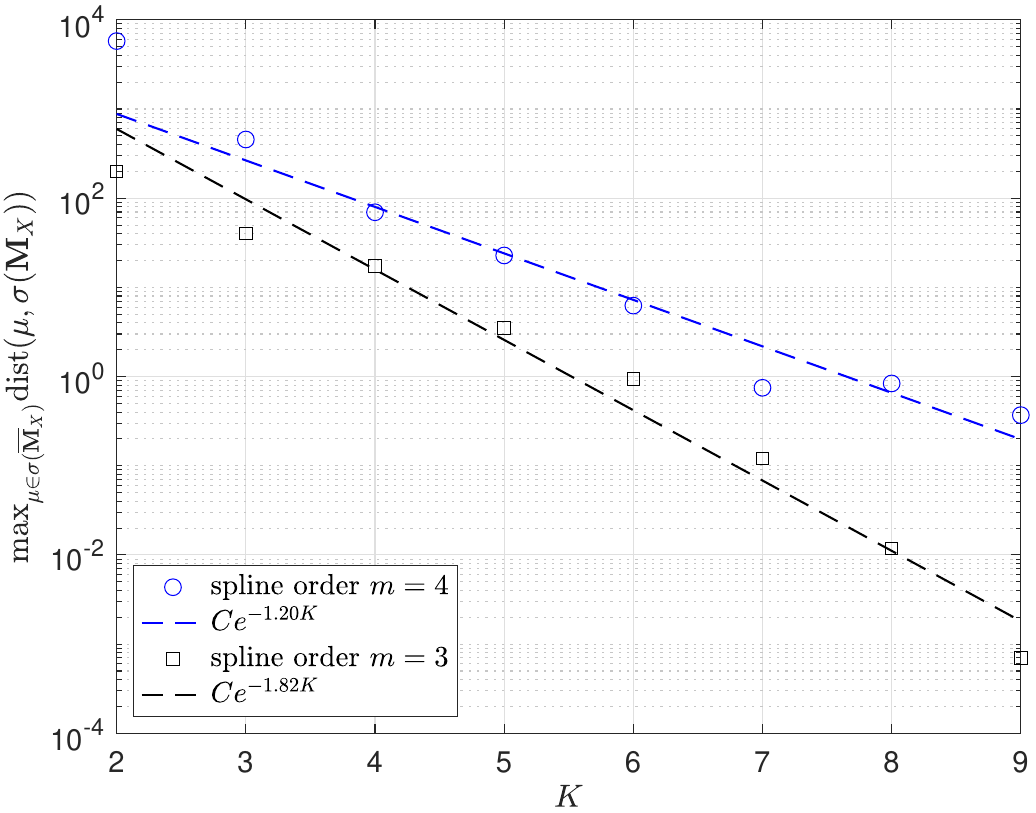} \\
(a) Fibonacci & (c) Maximum determinant \\
\includegraphics[width=0.47\textwidth]{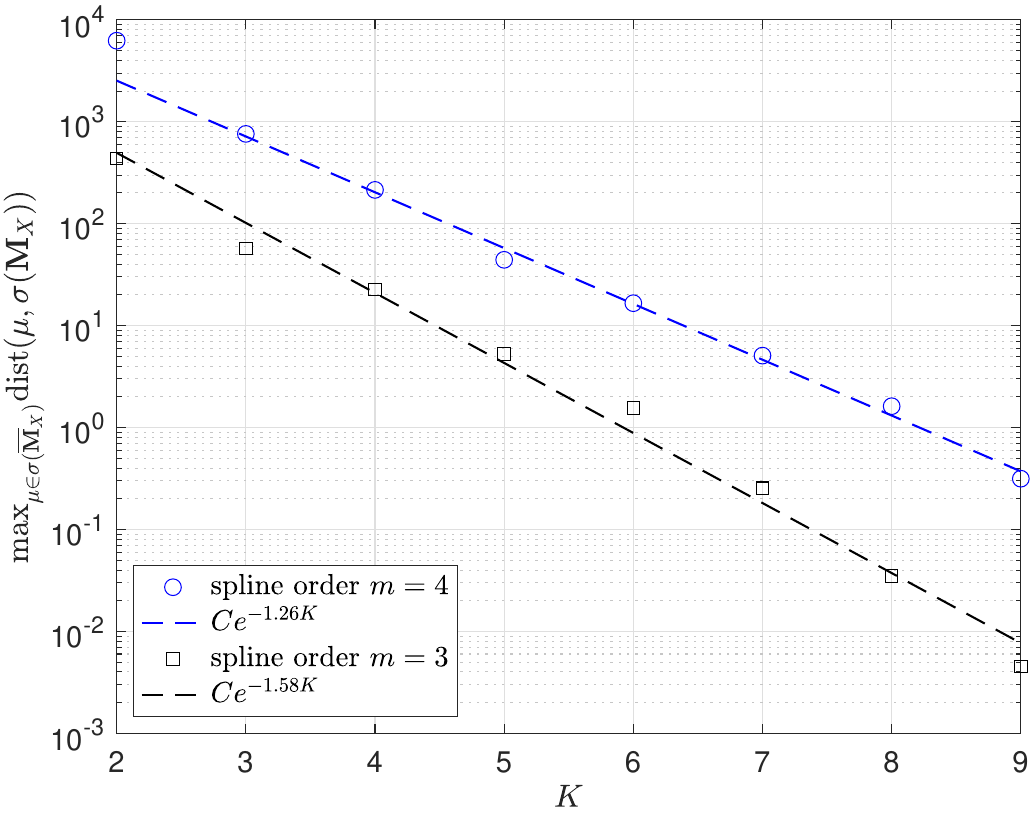} & \includegraphics[width=0.47\textwidth]{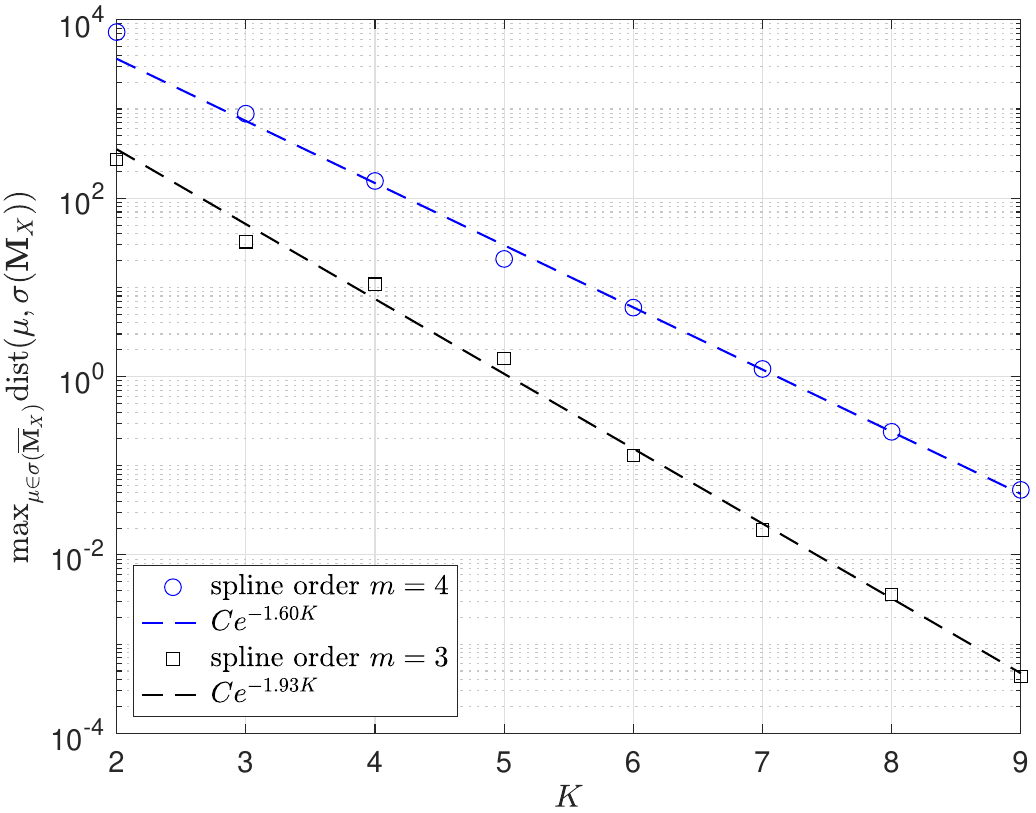} \\
(b) Minimum Energy & (d) Hammersley
\end{tabular}
\caption{Numerical results comparing the maximum of the minimum distances between the spectra of the global DM ($\MM_X$) and local Lagrange DM ($\MLL_X$) as a function of the parameter $K$ controlling the stencil size of the local Lagrange basis.  Results are for $\opL=\Delta$ computed from the restricted surface splines \eqref{surface_spline} of order $m=3$ and $m=4$ with augmented spherical harmonic spaces $\Pi_{m-1}$.  Dashed lines are the lines of best fit (on a log-linear scale) to the data for $3 \leq K \leq 9$, which indicate an exponential decay rate with increasing $K$.\label{fig:spectra_loclag}}
\end{figure}

The results from the previous section 
show that
\begin{align}
\max_{\mu\in\sigma(\MLL_X)}\dist\bigl(\mu, \sigma(\MM_X)\bigr) = \max_{\mu \in \sigma(\MLL_X)}\min_{\mu^\star \in \sigma(\MM_X)} |\mu-\mu^\star|
\label{eq:spectra_dist}
\end{align} 
can be controlled by an expression of the form
$h^{\alpha K +\beta}$ where $\alpha>0$, although the precise dependence of $\alpha$ and $\beta$
on $m$ and $\opL$ are not
easily determined analytically. 
In this section, we numerically examine the exponential decay of these bounds in terms of $K$ for $\opL = -\Delta$.
Rather than choosing the support points in each stencil from a ball search with a radius that depends on $K$, we use a nearest neighbor search so that the cardinality of each stencil is fixed at
\begin{align}
n = \lceil \frac17 K^2 (\log N)^2\rceil.
\label{eq:stencil_support}
\end{align}
For a quasi-uniform point set $X$, this gives similar results to a ball search.  Additionally, this nearest neighbor approach with a fixed stencil size is more common in RBF-FD methods~\cite{fornberg_flyer_2015}.

We again use the four families of point sets described in section \ref{S:Prelim} (and displayed in 
Figure \ref{fig:node}): Fibonacci, minimum energy, maximum determinant, and Hammersley.  As discussed in that section, the first three of these are quasi-uniform so the theory from the previous section applies.  We also include results on the Hammersley points to see if this theory can potentially be generalized to more general point sets. For the numerical experiments we set $N=4096$ for all but the Fibonacci nodes, where $N=4097$ since they are only defined for odd numbers. 
\begin{figure}[t]
\begin{tabular}{cc}
\includegraphics[width=0.48\textwidth]{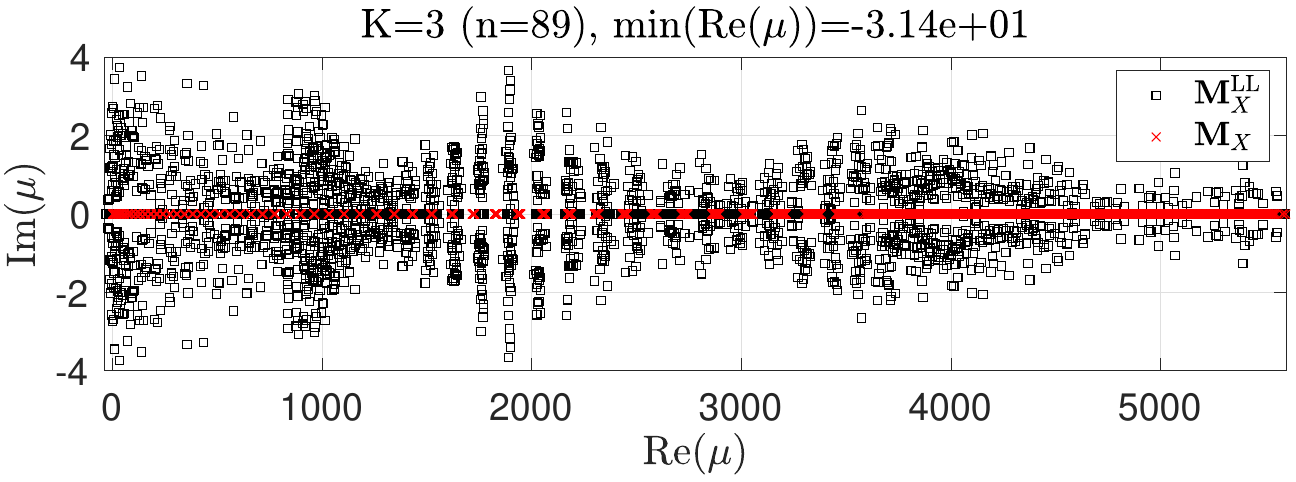} & \includegraphics[width=0.48\textwidth]{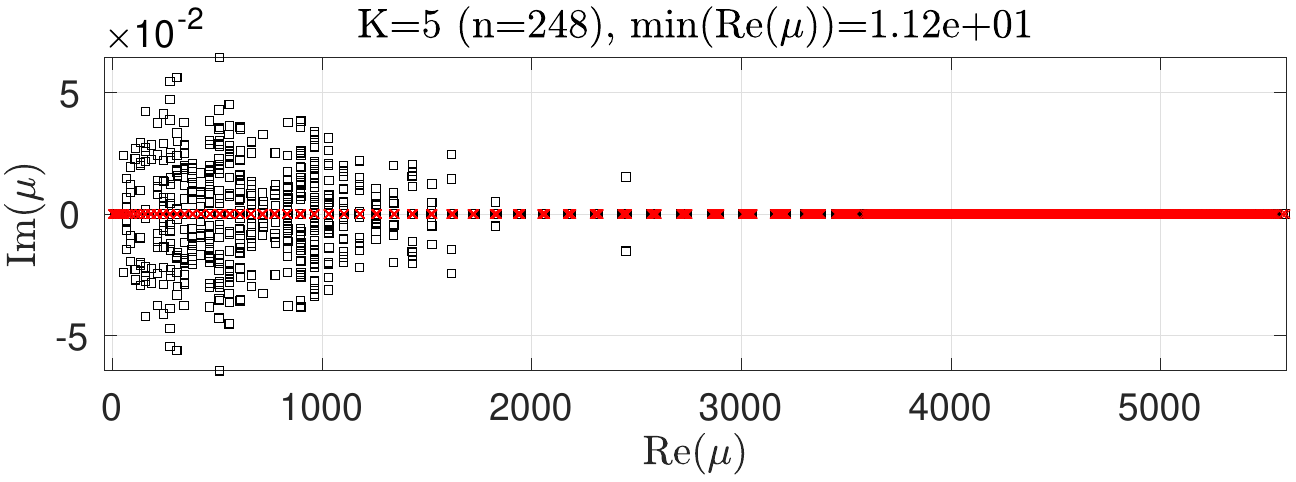} \\
\includegraphics[width=0.48\textwidth]{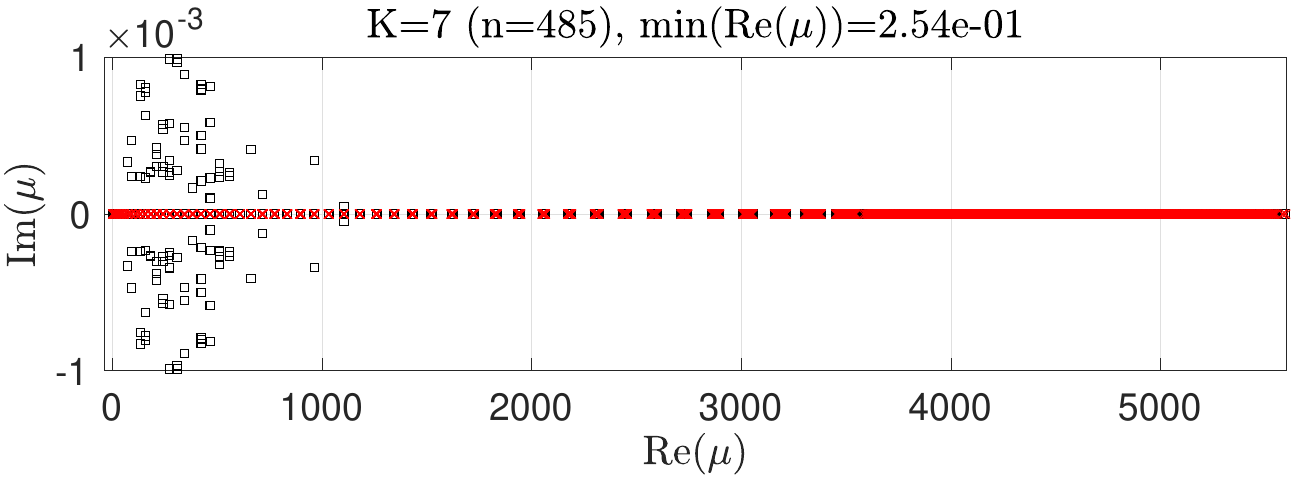} & \includegraphics[width=0.48\textwidth]{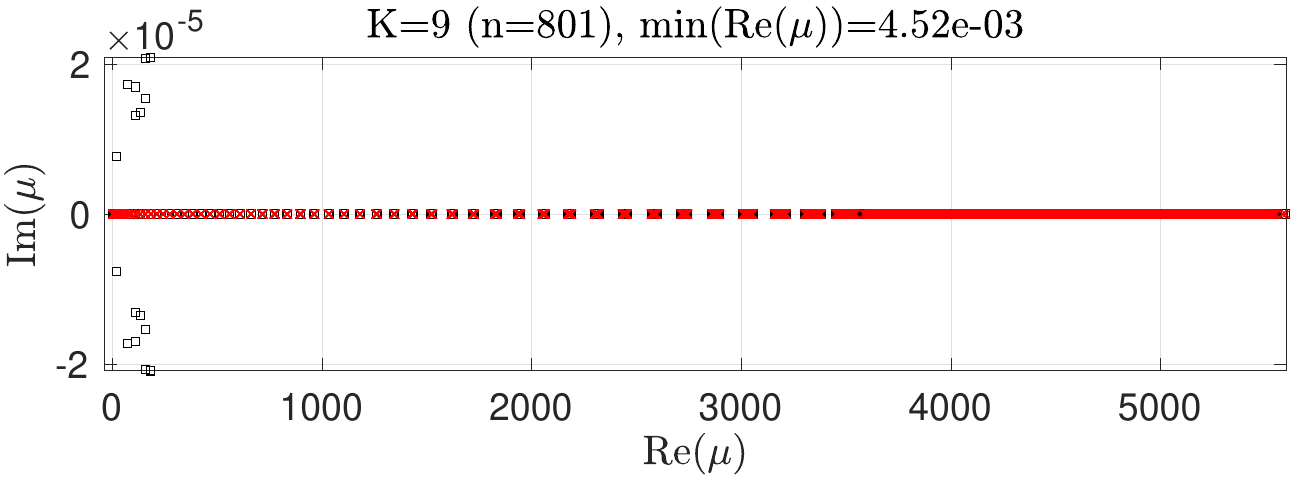} \\
\end{tabular}
\caption{Comparison of the spectra of the global ($\MM_X$) and local Lagrange ($\MLL_X$) DMs based on different stencil sizes. Results are for the $\opL=\Delta$ with $N=4096$ minimum energy points and the kernel is the restricted surface spline kernel of order $m=3$.  The minimum real part of spectrum of $\MLL_X$ is listed in the title of each plot . Note the different scales for the vertical (imaginary) axes on each plot.\label{fig:spectra_comparison_loclag}}
\end{figure}

Figure \ref{fig:spectra_loclag} displays the results from the experiments for the restricted surface spline kernels of orders $m=3$ and $m=4$. We see from the figure that distance between the spectra \eqref{eq:spectra_dist} decreases exponentially fast with $K$ for all four families of point sets, but that the rate depends on the geometry of the point set. We also see that the decay rate depends on the order $m$ of the kernels as the estimates from the previous section predict.  

To better compare the spectra of $\MM_X$ and $\MLL_X$, we plot the entire spectrum of each matrix for the minimum energy points and four different $K$ values in Figure \ref{fig:spectra_comparison_loclag}. We see from the plots in this figure that as $K$ increases the spread of the eigenvalues of $\MLL_X$ in the complex plane decreases.  Additionally, for this point set, the real part of the spectrum of $\MLL_X$ is non-negative for all but the $K=3$ case (see in the titles of the plots of the spectra).
%

%
%
%
\subsection{Numerical results on spectra of the RBF-FD DMs}\label{rbf_fd}
\begin{figure}[htb]
\centering
\begin{tabular}{cc}
\includegraphics[width=0.47\textwidth]{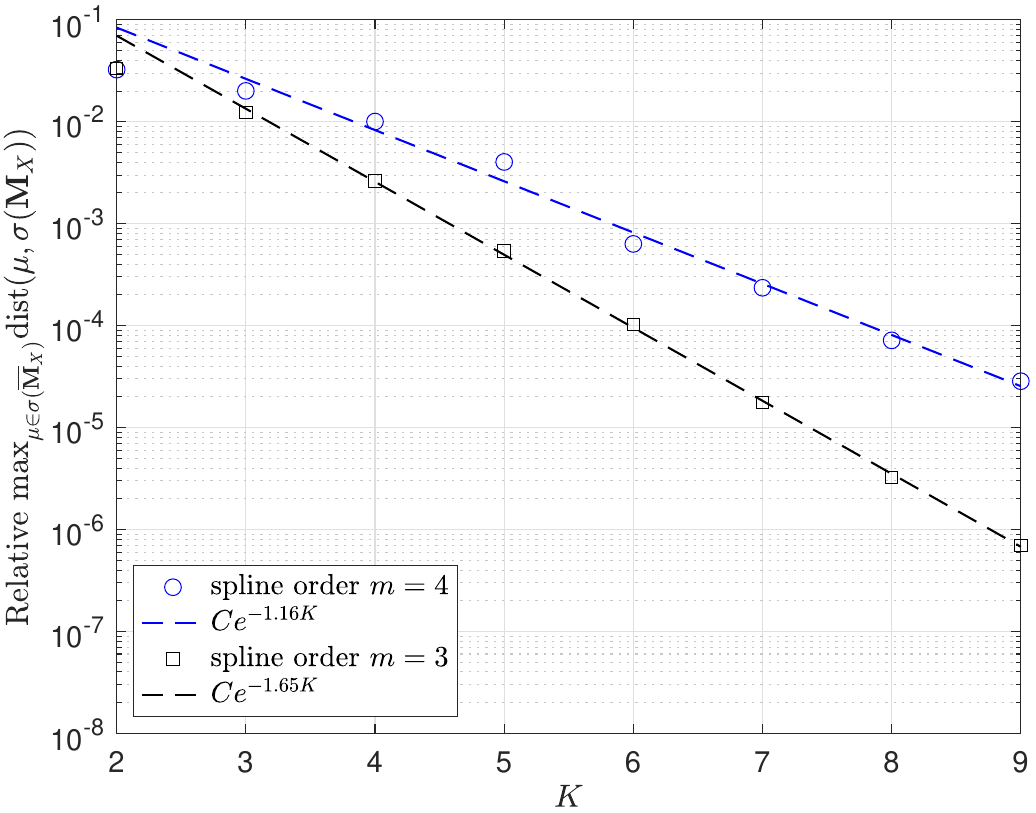} & \includegraphics[width=0.47\textwidth]{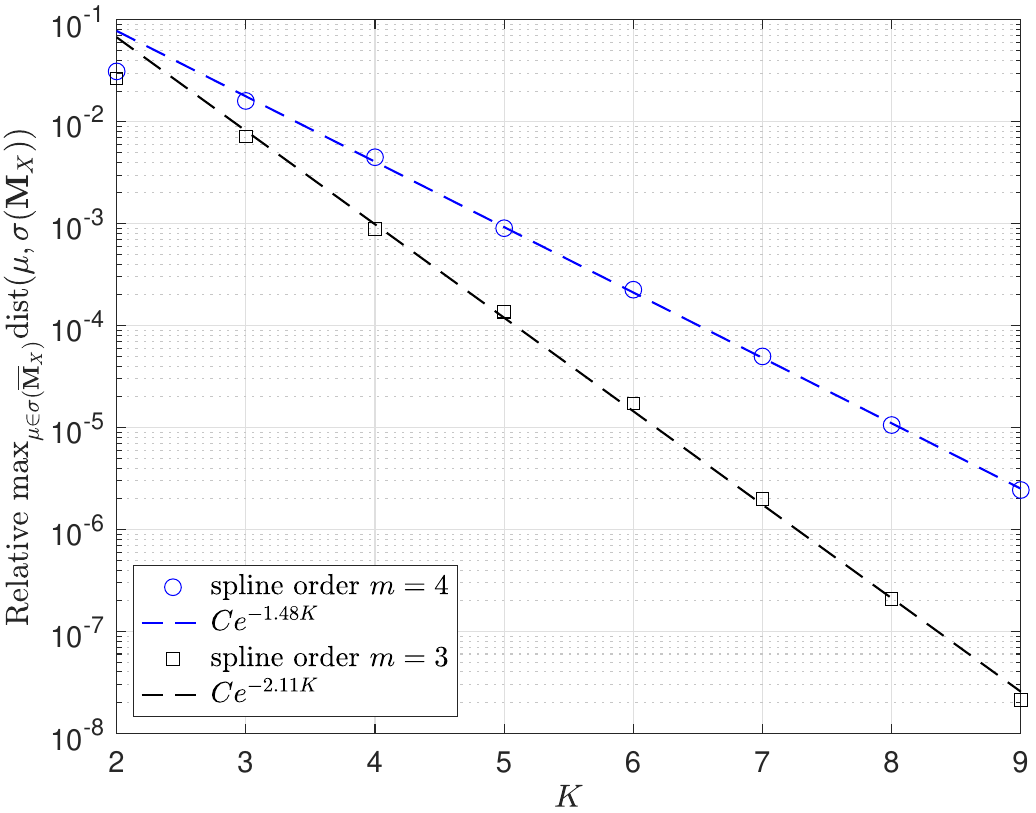} \\
(a) Fibonacci & (c) Maximum determinant \\
\includegraphics[width=0.47\textwidth]{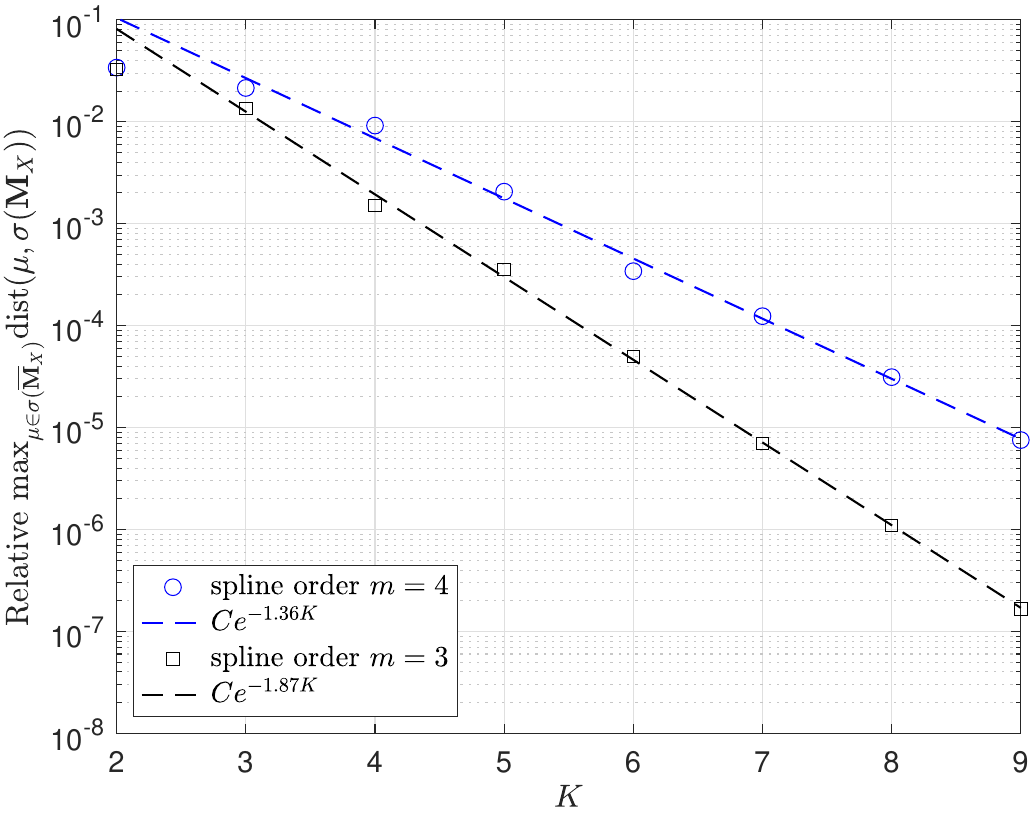} & \includegraphics[width=0.47\textwidth]{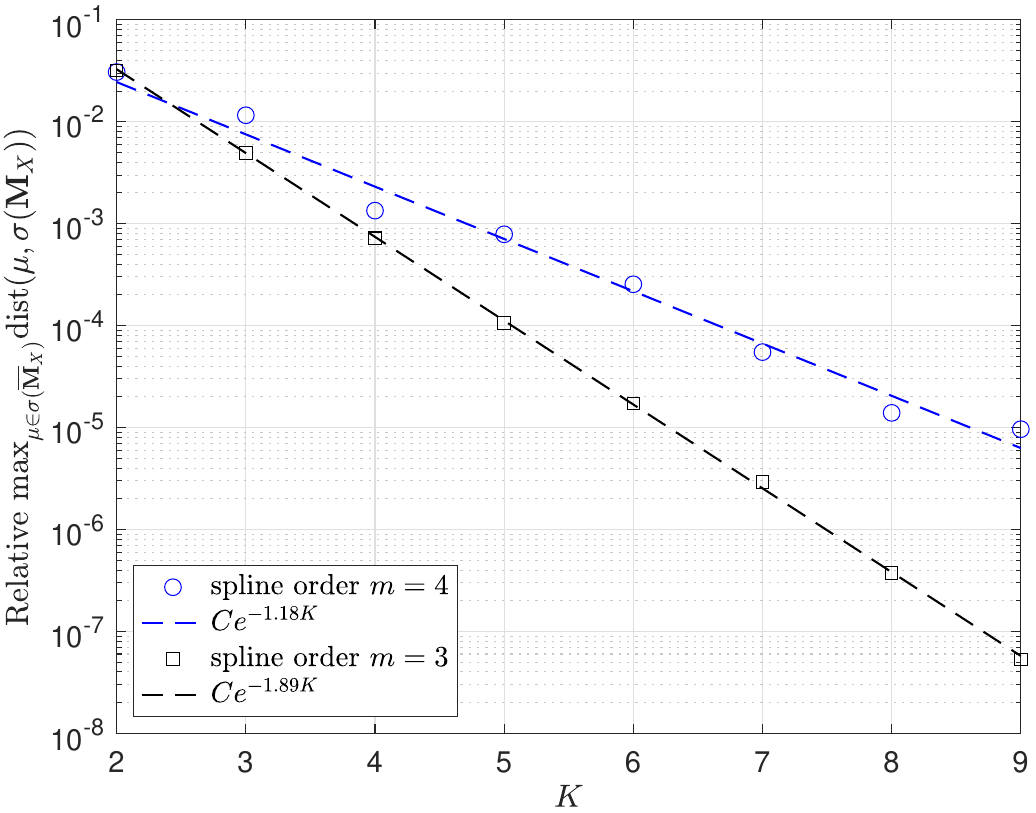} \\
(b) Minimum Energy & (d) Hammersley
\end{tabular}
\caption{Same as Figure \ref{fig:spectra_loclag}, but for the RBF-FD DM ($\MFD_X$).\label{fig:spectra_rbffd}}
\end{figure}

The RBF-FD method is similar to the local Lagrange method in that it produces a sparse DM, $\MFD_X$, for approximating $\opL$ at a set of points $X$ using kernel interpolation over local stencils of points.  However, unlike the local Lagrange case, estimates are not yet available for bounding $\|\MM_X-\MFD_X\|$ so that the results of Proposition 
\ref{Generalized_BF} or Theorem \ref{Generalized_BF_Diagonalized} can be applied to the spectral stability of $\MFD_X$ .  In this section we numerically compare the stability of $\MFD_X$ in terms of $\MM_X$ using \eqref{eq:stencil_support} to give evidence that similar bounds to the ones from Section \ref{S:Surface_splines} for $\MLL_X$ hold.

First, we briefly review the RBF-FD method for the case of surface splines on $\Sph^2$; see ~\cite{fornberg_flyer_2015} for more general settings and details.  Let $\Upsilon_j\subset X$ again denote a local stencil of points for each $x_j \in X$, and $u:\Sph^2\rightarrow\R$ be some generic function. To approximate the operator $\opL u$ at $x_j$, the RBF-FD method uses
\begin{align*}
\opL u(x_j) \approx \sum_{k\in\sigma_j} \opL\chi_k^{(j)}(x_j)  u(x_k),
\end{align*}
where $\chi_k^{(j)}$ are \emph{stencil Lagrange functions} associated with $\Upsilon_j$ and $\sigma_j$ is the index set for the stencil $\Upsilon_j$ in the global node set $X$.  The stencil Lagrange functions are in the space $S_{\Upsilon_j}$ defined in \eqref{trialspace} and satisfy $\chi_k^{(j)}(x_i) = \delta_{i,k}$ for $i,k\in\sigma_j$.  The entries of RBF-FD DM are then given as follows:
$$
\Bigl(\MFD_X\Bigr)_{j,k} =\begin{cases} \opL \chi_k^{(j)}(x_j), & k \in\sigma_j \\ 0,& \text{otherwise}. \end{cases}
$$
The difference between the local Lagrange and RBF-FD methods may be subtle, but it is important.  The latter is based on a separate interpolant over the stencil $\Upsilon_j$ defined by the target point $x_j$, while the former is based on interpolants defined over all stencils that have a non-empty intersection with the target point $x_j$.  The RBF-FD method is then necessarily exact on the space $\Pi_m(X)$, but the local Lagrange method is not. We finally note that both methods produce $\MM_X$ in the limit that each stencil includes every point in $X$, i.e., $\Upsilon_j=X$.

For the numerical results, we follow exactly the same set-up as the previous section and compare the spectra of $\MFD_X$ to $\MM_X$ using the measure \eqref{eq:spectra_dist} for different $K$ and surface spline orders $m$. We note that since $\MFD_X$ is exact on $\Pi_m(X)$ its spectrum will included $\Lambda_{m-1} = \{\lambda_{\ell}\}_{\ell = 0}^{m-1}$, where are the first $m-1$ eigenvalues of $\opL$.

\begin{figure}[t]
\begin{tabular}{cc}
\includegraphics[width=0.48\textwidth]{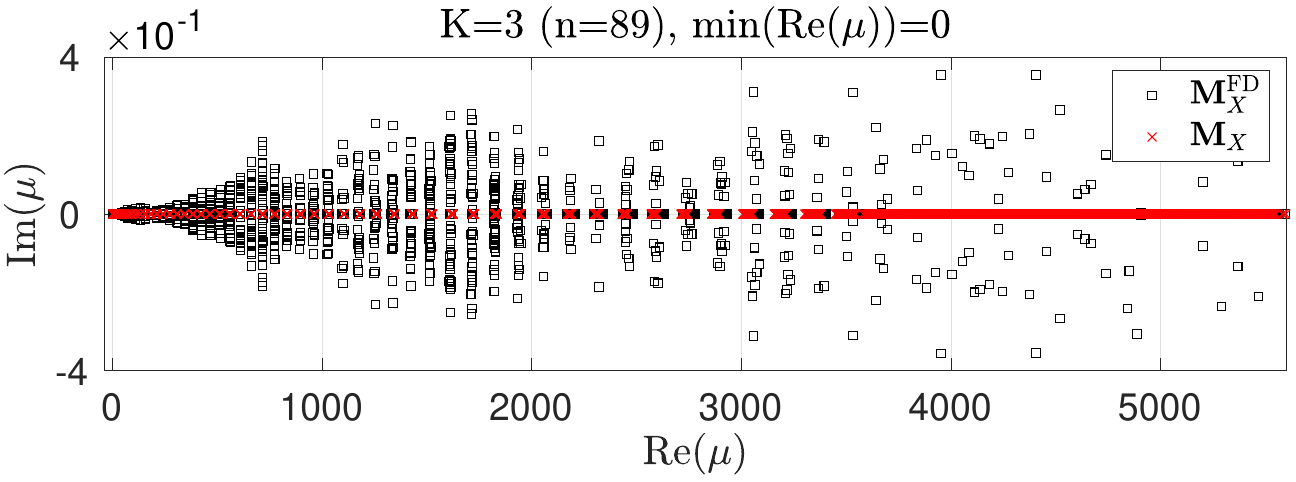} & \includegraphics[width=0.48\textwidth]{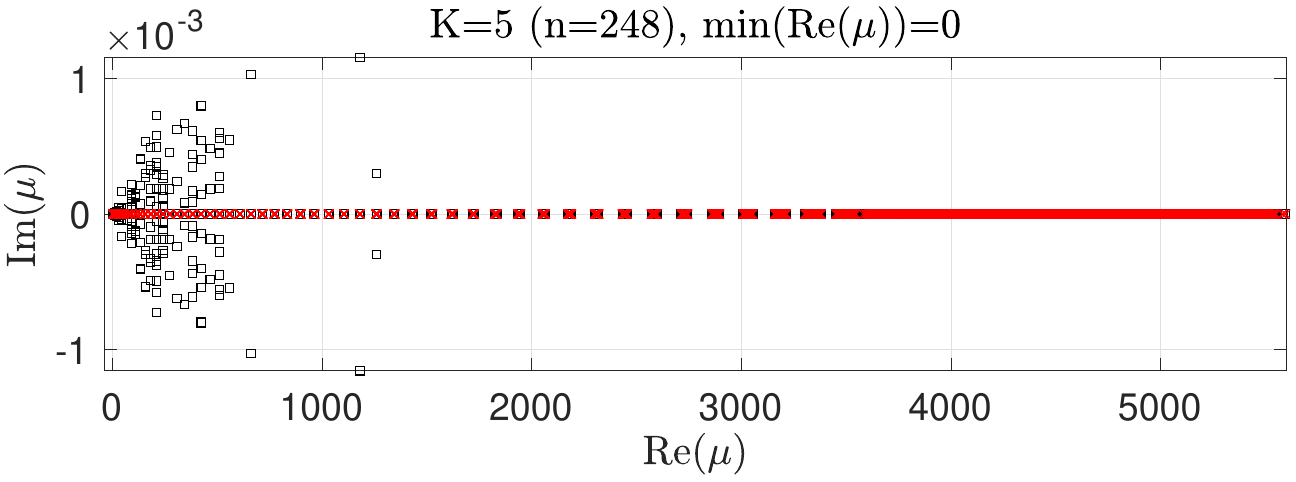} \\
\includegraphics[width=0.48\textwidth]{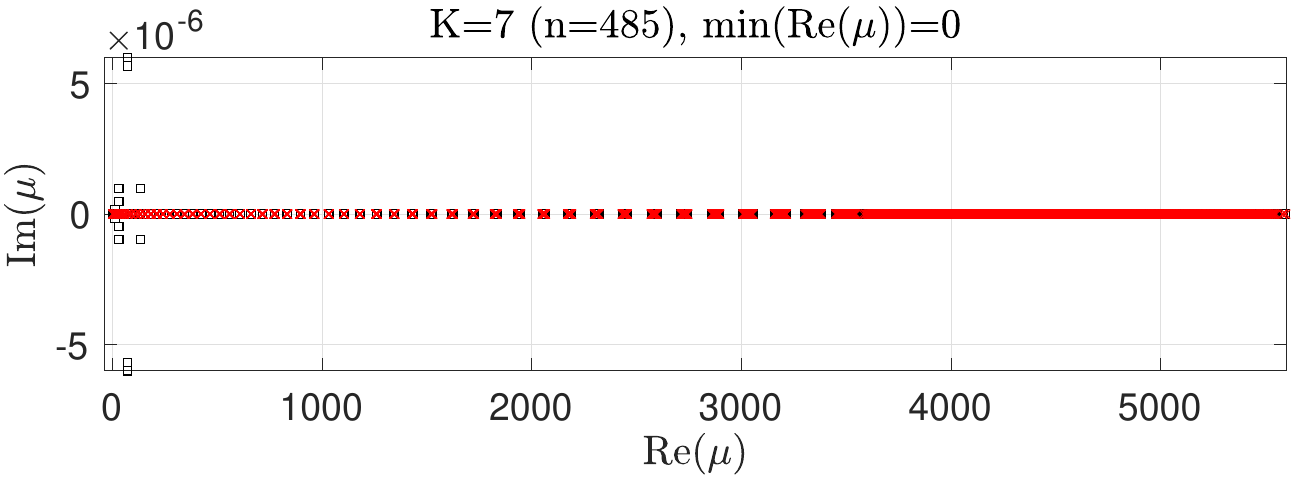} & \includegraphics[width=0.48\textwidth]{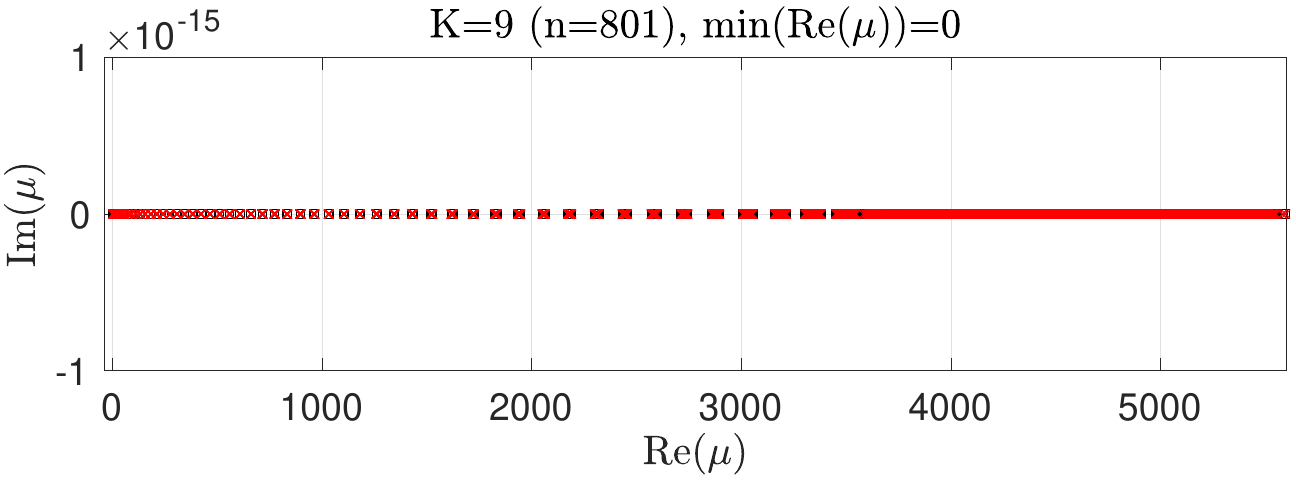} \\
\end{tabular}
\caption{Same as Figure \ref{fig:spectra_comparison_loclag}, but for the RBF-FD DM ($\MFD_X$).\label{fig:spectra_comparison_rbffd}}
\end{figure}

Figure \ref{fig:spectra_rbffd} displays the results from the experiments.  Similar to the local Lagrange case, we see from the figure that the distance between the spectra \eqref{eq:spectra_dist} for $\MFD_X$ also decreases exponentially fast with $K$ for all four families of point sets. This indicates that similar bounds to those from Section \ref{S:Surface_splines} for $\MLL_X$ may also hold for $\MFD_X$.

The spectra of $\MM_X$ and $\MFD_X$ are directly compared in Figure \ref{fig:spectra_comparison_rbffd} for the minimum energy points. We see from this figure that for the same $K$ the spectrum of $\MFD_X$ does not extend as far in the complex plane as the spectrum of $\MLL_X$ given in Figure \ref{fig:spectra_comparison_loclag}. Additionally, the rate at which the spectrum approaches the positive real axis as $K$ increases is higher for the RBF-FD method.  Finally, the figure shows that for each $K$ the eigenvalue of $\MFD_X$ with the smallest real part is 0, which is the expected value for $\opL = -\Delta$.

\section{Concluding remarks}
The results of this paper have addressed important issues regarding the spectrum of DMs that arise in global and local kernel collocation methods and their stability under perturbations.  However, there remain significant open questions in both the global and the local cases. We briefly note three interesting open problems: 
\begin{itemize}
\item There is more to temporal stability of method of lines (or semi-discrete) approximations like \eqref{eq:semi_discrete} considered in section \ref{SS:energy_stability} than the spectrum of the DM $\MM_X$ or the energy stability of the solutions of \eqref{eq:semi_discrete}.  Showing stability for the fully discrete problem (i.e., after discretizing \eqref{eq:semi_discrete} in time) requires addressing questions about the pseudospectra (resolvent stability)~\cite{TrefethenEmbree2005} or strong stability~\cite{Tadmor} of the DMs.  In the latter case, \cite[Section 3.1]{Tadmor} considers the coercivity condition
$$\Re\langle u,\MM_X u\rangle_* \le -\eta\|\MM_X u\|_*^2,\; \eta > 0, $$
 of a matrix in some inner product  $\langle\cdot,\cdot\rangle_*$ (this is\cite[(3.2)]{Tadmor}).  Coercivity guarantees stability of certain Runge-Kutta method applied to \eqref{eq:semi_discrete} provided a CFL-type condition holds. 
\item If $\sigma(\opL)\subset(-\infty,0]$ and $0\in\sigma(\opL)$,  then
a perturbed DM $\MM^{\epsilon}$ may fail to inherit the Hurwitz stability of $\MM_X$.
In short, the results presented here do not guarantee that perturbation  preserves
weak Hurwitz stability (semi-stability).
\item 
The challenge of obtaining useful perturbation errors for $\MFD_X$ that are similar to (\ref{perturbation_bound}) amounts to controlling
the difference $ \opL \chi_k^{(j)}(x_j)-\opL\chi_j(x_j)$.
Superficially, this may seem similar to 
$\opL b_j(x_j)-\opL\chi_j(x_j)$, but the challenge 
stems from the fact that the center $x_k$ associated with
$\chi_k^{(j)}$ may be 
situated near to the boundary of the domain of
$B(x_j,Kh|\log h|)$, which is
 where the  decay conditions of the Lagrange
functions are not yet well understood.
\end{itemize}

\bibliographystyle{plain}
\bibliography{gfd}
\end{document}